\documentclass{amsart} 

\usepackage[T1]{fontenc}
\usepackage[utf8]{inputenc}
\usepackage{newtxtext}

\usepackage[a4paper, margin=30mm, headheight=14pt]{geometry}

\usepackage{mathtools} 
\usepackage{amsthm,amssymb}

\usepackage{enumitem}
\usepackage{booktabs}
\usepackage{graphicx}
\usepackage{xcolor}
\usepackage{fancyhdr}

\usepackage{aliascnt}

\usepackage{hyperref}

\usepackage[nameinlink,capitalize,noabbrev]{cleveref}

\DeclarePairedDelimiter{\abs}{\lvert}{\rvert}
\DeclarePairedDelimiter{\norm}{\lVert}{\rVert}

\DeclarePairedDelimiterX{\inner}[2]{\langle}{\rangle}{#1,#2}

\newcommand{\R}{\mathbb{R}}
\newcommand{\C}{\mathbb{C}}

\newcommand{\Z}{\mathbb{Z}}
\newcommand{\PP}{\mathcal{P}}
\newcommand{\HH}{\mathbb{H}}

\newcommand{\dd}{\mathop{}\!\mathrm{d}}
\newcommand{\e}{\mathrm{e}}
\newcommand{\ii}{\mathrm{i}}

\DeclareMathOperator{\tr}{tr}

\DeclareMathOperator{\Hopf}{Hopf}

\DeclareMathOperator{\diam}{diam}
\DeclareMathOperator{\Fix}{Fix}
\DeclareMathOperator{\aarg}{arg}
\DeclareMathOperator{\pr}{Pr}

\newcommand{\re}{\operatorname{Re}}
\newcommand{\im}{\operatorname{Im}}

\numberwithin{equation}{section}

\newtheoremstyle{plain-nopunct}
  {\topsep}
  {\topsep}
  {\itshape}
  {}
  {\bfseries}
  {}
  {0.5em}
  {}

\newtheorem{theorem}{Theorem}[section]

\newaliascnt{lemma}{theorem}
\newtheorem{lemma}[lemma]{Lemma}
\aliascntresetthe{lemma}

\newaliascnt{proposition}{theorem}
\newtheorem{proposition}[proposition]{Proposition}
\aliascntresetthe{proposition}

\newaliascnt{corollary}{theorem}
\newtheorem{corollary}[corollary]{Corollary}
\aliascntresetthe{corollary}

\newaliascnt{conjecture}{theorem}

\aliascntresetthe{conjecture}

\newaliascnt{definition}{theorem}
\newtheorem{definition}[definition]{Definition}
\aliascntresetthe{definition}

\newaliascnt{example}{theorem}
\newtheorem{example}[example]{Example}
\aliascntresetthe{example}

\newaliascnt{question}{theorem}

\aliascntresetthe{question}

\newaliascnt{remark}{theorem}
\newtheorem{remark}[remark]{Remark}
\aliascntresetthe{remark}

\crefname{theorem}{Theorem}{Theorems}
\crefname{lemma}{Lemma}{Lemmas}
\crefname{proposition}{Proposition}{Propositions}
\crefname{corollary}{Corollary}{Corollaries}
\crefname{conjecture}{Conjecture}{Conjectures}
\crefname{definition}{Definition}{Definitions}
\crefname{example}{Example}{Examples}
\crefname{question}{Question}{Questions}
\crefname{remark}{Remark}{Remarks}
\crefname{appendix}{Appendix}{Appendices}

\title[Harmonic maps with Scherk asymptotics]{Harmonic Maps from Punctured Riemann Surfaces to the Hyperbolic Plane with Prescribed Scherk Asymptotics}

\author{Qiongling Li}
\address{
Chern Institute of Mathematics and LPMC,
Nankai University,
Tianjin 300071, China
}
\email{qiongling.li@nankai.edu.cn}

\author{Jinsong Liu}

\author{Yihui Xu}
\address{
State Key Laboratory of Mathematical Sciences, AMSS, Chinese Academy of Sciences, Beijing 100190, China; School of Mathematical Sciences, University of Chinese Academy of Sciences, Beijing 100049, China
}
\email{liujsong@math.ac.cn \qquad xuyihui@amss.ac.cn}
\thanks{Corresponding author: Yihui Xu.}


\begin{document}

\begin{abstract}
    Let $X$ be a genus $g\geq 0$ compact Riemann surface that admits an antiholomorphic involution $\iota$ with nonempty fixed-point set, and let $D = \{p_1,\dots,p_k\}\subset \Fix(\iota)$. At each puncture, we prescribe a Scherk map associated to a given real polynomial quadratic differential with even degree and negative leading coefficient. Here a Scherk map is the harmonic diffeomorphism from $\C$ to the interior of an ideal polygon in $\HH^2$, whose Hopf differential is the given polynomial. We construct a harmonic map
    \[
        h:X\backslash D \to \HH^2
    \]
    whose asymptotic behavior at each puncture matches the prescribed Scherk map. In particular, the image of $h$ tends to an ideal polygon near each end. The resulting $h$ covers harmonic maps obtained by taking the $\HH^2$ factors of horizontal catenoids in $\HH^2 \times \R$. 
\end{abstract}

\maketitle

\tableofcontents

\section{Introduction}\label{sec:introduction}
Harmonic maps between surfaces have been studied extensively, since \cite{SchoenYau1978Univalent}. One remarkable result is from Han-Tam-Treibergs-Wan \cite{HanTamTreibergsWan1995}. They proved that a nonconstant polynomial quadratic differential on the complex plane determines, up to an isometry of $\HH^2$, an orientation-preserving harmonic diffeomorphism from the complex plane onto the interior of an ideal polygon in $\HH^2$. Following \cite{Huang2016Harmonic} , we refer to such maps as \emph{Scherk maps}. 

In \cite{Wolf1989Teichmuller}, Wolf showed that the Teichm{\"u}ller space of genus $g$ closed surface $\mathcal{T}_g$ is homeomorphic to the space of holomorphic quadratic differentials $Q(X)$, using harmonic diffeomorphisms and their Hopf differentials. Here $X$ is a fixed genus $g$ Riemann surface. This gives a global parameterization of $\mathcal{T}_g$. Lohkamp extended this harmonic maps parameterization to punctured surfaces. In \cite{Lohkamp1991}, the Teichm{\"u}ller space of punctured surfaces is identified with the space of holomorphic quadratic differentials on a fixed Riemann surface that have at most simple poles at punctures, via the harmonic diffeomorphisms that have finite energy and their Hopf differentials.

Gupta established a wild analogue of Wolf's parameterizations, allowing poles of higher order and infinite energy. Suppose $X$ is a genus $g$ compact Riemann surface and $D\subset X$ is finite. In \cite{Gupta2021Harmonic}, Gupta proves that, given a marked crowned surface $Y$ that is homeomorphic to $X\backslash D$, and given compatible principal parts at punctures, there exists a unique harmonic diffeomorphism from $X \backslash D$ to $Y$ whose Hopf differential has precisely those principal parts at the punctures. Consequently, he obtained a homeomorphism between the corresponding Teichm{\"u}ller space of crowned hyperbolic surfaces and a space of meromorphic quadratic differentials with the prescribing principal part. 

In both results stated above, the harmonic maps under consideration are diffeomorphisms, and thus their domains and targets have the same topological type. A different class of examples, in which the domain and the target are not homeomorphic, arises from the projection of the minimal surfaces in $\HH^2 \times \R$ to the first factor. In \cite{Pyo2011NewComplete}, Pyo constructed a class of embedded minimal surfaces into $\HH^2 \times \R$ which is conformally a punctured sphere (the number of punctures is at least two). In particular, he constructed conformal minimal embeddings from $\C^*$ to $\HH^2 \times \R$, which is called horizontal catenoids. Such catenoids are also constructed by Morabito and Rodr{\'i}guez in a different way in \cite{MorabitoRodriguez2012}. In \cite{MartinMazzeoRodriguez2014}, using horizontal catenoids, Mart{\'i}n, Mazzeo, and Rodr{\'i}guez constructed complete properly embedded minimal surfaces in $\HH^2 \times \R$ with finite total curvature, genus $g$ and $k$ ends, provided $k$ is larger than a constant depending on $g$. By construction, the entire surface, as well as each of its ends, is invariant under the reflection across $\HH^2 \times \{0\}$. The induced metric endows the surface with a complex structure with respect to which the reflection is an antiholomorphic involution. Their construction gives rise, upon projection onto the $\HH^2$-factor, to a class of harmonic maps. Such a map is from a surface $\Sigma$ to $\HH^2$, where $\Sigma$ is conformally equivalent to a genus $g$ compact Riemann surface with $k$ punctures, and all these $k$ punctures lie in the fixed-point set of an antiholomorphic involution. The Hopf differential of this harmonic map has a global square root, and has a pole of order $4$ at each of the $k$ punctures. Coskunuzer \cite{Coskunuzer2021ArbitraryTopology} proved that, an arbitrary open orientable surface can be properly embedded into $\HH^2 \times \R$ as an area-minimizing complete surface. However, it does not prescribe the asymptotic behavior at its ends or guarantee finite total curvature.

In this paper, we construct new harmonic maps from punctured surfaces to $\HH^2$ with prescribed asymptotic behavior at each of the punctures. The main result is the following:

\begin{theorem}[Main theorem]
\label{thm:main}
    Let $X$ be a compact Riemann surface that admits an antiholomorphic involution $\iota$ with $\Fix(\iota)\neq\varnothing$, and let $D=\{p_1,\ldots,p_k\}\subset \operatorname{Fix}(\iota)$. For each $i=1,\cdots,k$, let $\Phi_i=P_i(z)\,dz^2$ be a nonzero polynomial quadratic differential of even degree, with real coefficients and negative leading coefficient, and let $\omega_i:\C \rightarrow \HH^2$ be the corresponding normalized Scherk map. Its image is the interior of an ideal polygon $\PP_i$. 
    
    Then there exists a harmonic map $h:X\backslash D\rightarrow\HH^2$ such that, there is a constant $C>0$ satisfying
    \[
        d_{\HH^2}\Big(h(p),\omega_i\big(\eta_i(p)\big)\Big)\leq C
    \]
    for every $i$ and for all $p$ sufficiently close to $p_i$. Here $\eta_i$ is a suitably chosen coordinate map on a punctured neighborhood of $p_i$ (see \cref{sec:existence}). It is unique if the union of images of all $\omega_i$ is not contained in a geodesic in $\HH^2$.
    
    Consequently, $\Hopf(h)$ extends meromorphically across each $p_i$ with pole order $\deg(P_i)+4$ at $p_i$.
\end{theorem}
We impose no restrictions on $g$ or $k$, while allowing greater flexibility in the prescribed asymptotic behavior of $h$ at punctures. The theorem is valid for the Scherk maps associated with negative constant quadratic differentials. We refer to a complete geodesic in $\HH^2$ as a degenerate $2$-gon. Note that $h$ induces a trivial homomorphism on the fundamental group. It remains a question whether these maps give local models of handle-crushing harmonic maps from higher genus surfaces to lower genus surfaces.

\begin{remark}\label{rem:comparisonwithcatenoids}
    Our result covers harmonic maps obtained by taking the $\HH^2$ factor of horizontal catenoids, which have been mentioned above. Suppose 
    \[
        \Psi = (F,h) :\C^* \to \HH^2\times \R
    \]
    is a horizontal catenoid. Let $X=\C \mathbb{P}^1, D=\{p_1=\infty, p_2=0\}$. By construction, the catenoid is asymptotic to the vertical plane $\gamma_i \times \R$ at the end near $p_i$, where $\gamma_i$ is a complete geodesic for each $i=1,2$. Let $\omega_i$ be the harmonic map from $\C$ to $\HH^2$ with image $\gamma_i$ obtained by composing the map $\C\rightarrow \R (z\mapsto \im(z))$ and the $2a$-length parametrization $\R\rightarrow \gamma_i$. The map $\omega_i$ has constant Hopf differential $-a^2\dd z^2$ for some constant $a>0$ depending on $\Psi$, and the correspondence between the horizontal asymptotic directions of $\omega_i$ and ideal vertices of $\gamma_i$ matches with that of $F$. Then the resulting $h$ in \cref{thm:main} is exactly $F$. See \cref{sec:horizontalcatenoids}.
\end{remark}

While the proof is inspired by the ideas in \cite{Huang2016Harmonic}, we give a separate argument here to keep the presentation completely self-contained. The strategy of the proof is as follows. We solve a family of Dirichlet boundary value problems
\begin{equation*}
    \left\{
    \begin{aligned}
        &h^s:\Sigma_s \rightarrow \mathbb{H}^2,\\
        &h^s|_{\partial\Sigma_s}= \omega|_{\partial\Sigma_s}.
    \end{aligned}
    \right.
\end{equation*}
Here $\omega$ is the prescribed Scherk map near the punctures and $\{\Sigma_s\}$ is compact exhaustion of $X \backslash D$. To obtain the precompactness of the resulting family of solutions $\{h^s\}$ on $\Sigma_r$ for fixed $r$, it suffices to establish uniform energy bounds on $\Sigma_r$. The key ingredient in deriving these energy bounds is the boundedness of distance between parachute maps and Scherk maps, which will be introduced in \cref{sec:boundedness}. The assumptions on the antiholomorphic involution $\iota$ play an important role in the proof, as they ensure that the family of solution is bounded at a point. In order to compare the resulting map $h$ with the $\HH^2$ factor of a horizontal catenoid, we establish a principal part criterion of bounded distance between harmonic maps in \cref{sec:principal-bound}, and make full characterization of the induced metric on horizontal catenoids on a chosen coordinate in \cref{prop:isometrybyhopf}. Consequently, $h$ is exactly the $\HH^2$ factor of a horizontal catenoid if the asymptotic data is chosen in a proper way. 

\subsection*{Organization of the paper} \cref{sec:preliminaries} reviews harmonic maps, Hopf differentials, Scherk maps, and some estimates used below. In \cref{sec:equicontinuity} we derive local equicontinuity from uniform energy bounds on compact subsets. In \cref{sec:boundedness}, we establish the uniform bound for the distance between the parachute maps and Scherk map, then use the isoenergy inequality to bound the energy gap between a parachute map and its Scherk model. In \cref{sec:existence} we solve the exhausting Dirichlet problems, prove the asymptotic description at every puncture, and establish convergence of the full family and nondegeneracy of the limit. We also give a uniqueness result \cref{prop:unique}, provided that the image of the harmonic map is not contained in a geodesic. The proof of \cref{thm:main} is finished here. In \cref{sec:principal-bound}, we provide a principal part criterion of bounded distance between harmonic maps. In \cref{sec:horizontalcatenoids}, we characterize the induced metric on horizontal catenoids on a certain chosen coordinate and compare their $\HH^2$ factors with the resulting $h$ in \cref{thm:main}. An exponential decay estimate is established in \cref{app:decayestimate}. \cref{app:hess} contains the Hessian estimate used to prove the meromorphicity at the punctures. The polygonal exhaustion used in the proof of \cref{prop:principaltobounded} is illustrated in \cref{app:polygonexhaustion}. 

\subsection*{Acknowledgments}The authors would like to thank Michael Wolf for his helpful comments on this paper. Qiongling Li was partially supported by the National Key R\&D Program of China No. 2022YFA1006600, the Fundamental Research Funds for the Central Universities, Nankai Zhide foundation, and the Alexander von Humboldt Foundation. Jinsong Liu was supported by the National Key R\&D Program of China (Grant No. 2021YFA1003100), NSFC (Grant No. 12631004), and the Natural Science Foundation of Guangdong Province (Grant No. 2025A1515011213).

\subsection*{AI assistance}
The authors were originally unsuccessfully trying to prove \cref{coro:usboundoncircle} by showing the boundedness of Hopf differentials of parachute maps. That method works only for the special case of regular polygons. ChatGPT version 5.6 plus gave a new proof, that is \cref{prop:usomegabound}. In \cref{app:hess}, the proof of the meromorphicity and the exact pole order of $\Hopf(h)$ at punctures using the Hessian inequality was inspired by a guided conversation with ChatGPT version 5.6 plus. The original proof of \cref{prop:principaltobounded} by the authors was rather complicated. ChatGPT version 5.6 plus also helped to simplify it to the present version. The authors verified, corrected, and revised the proof from AI, and take responsibility for the final content.

\section{Preliminaries}
\label{sec:preliminaries}

Throughout the paper, $ \HH ^2$ denotes the hyperbolic plane and $\sigma$ denotes the hyperbolic metric on $\HH ^2$ with constant curvature $-1$. Unless otherwise specified, we always use the unit disk model for $\HH^2$. 

Suppose $X$ is a compact Riemann surface with genus $g \geq 0$. Let $\iota : X \to X$ be an antiholomorphic involution, and denote its fixed-point set by $\Fix(\iota)=\{p\in X:\iota(p)=p\}$. For example, on $\C\mathbb{P}^1$, the antiholomorphic involution $z \mapsto \bar{z}$ has a nonempty fixed-point set $\R \cup\{\infty\}$. Note that an antiholomorphic involution may not exist for a given Riemann surface, and even when it exists, $\Fix(\iota)$ may be empty. However, for each genus $g$, we can always find a Riemann surface of genus $g$ with such an involution with nonempty fixed-point set. In the following example, we present a class of hyperelliptic curves of any genus $g \geq 1$ admitting an antiholomorphic involution with a nonempty fixed-point set.

\begin{example}
    Let $X$ be the genus $g$ Riemann surface determined by
    \[
        w^2=(z-a_1)(z-a_2)\dots(z-a_{2g+2}),
    \]
    where $a_1<a_2<\dots<a_{2g+2}$ are distinct real numbers. Let $\iota$ be the antiholomorphic involution given by $(z,w)\mapsto (\bar{z},\bar{w})$. For  $i=1,2,\dots2g+2$, we have $(a_i,0) \in \Fix(\iota)$. 
\end{example}

From now on, we assume that $X$ admits an antiholomorphic involution $\iota$ with $\Fix(\iota)\neq\varnothing$.

\subsection{Harmonic maps}\label{subsec:harmonic-maps}
We recall some general concepts related to harmonic maps between two Riemannian manifolds $(M,g)$ and $(N,h).$ We further assume that $M$ is  a Riemann surface.

\begin{definition}\label{def:harmonic-map}
A $C^2$-smooth map $f\colon (M,g)\to (N,h)$  is called harmonic if it is a critical point of the energy functional
\[
    E(f,U)=\int_U \frac{1}{2} \norm*{\dd f}^2 \dd V_M
\]
on every relatively compact open subset $U \subset M$. The function in the integrand is called the energy density of $f$ and is denoted by $e(f)$. The Euler-Lagrange equation of the energy functional is
\[
    \tau(f)=0,
\]
where
$\tau(f):=\tr(\nabla \dd f)$
is the tension field of $f.$ Thus $f$ is harmonic if and only if its tension field vanishes identically. 
\end{definition}

\begin{definition}\label{def:hopf}
    For a $C^2$-smooth map $f\colon (M,g)\to (N,h)$, the Hopf differential of $f$ denoted by $\Hopf(f)$ is 
    \[
        \Hopf(f)=(f^*h)^{(2,0)},
    \]
    a quadratic differential on $M$ with the form $\Hopf (f) = \varphi(z)\dd z^2$ in local coordinates. It is well known that $\Hopf(f)$ is holomorphic if $f$ is harmonic.
\end{definition}    

From now on, let $(N,h)=(\HH ^2,\sigma)$. Write $\sigma=\sigma(z)|\dd z|^2$. Let $f \colon M\to \HH^2$ be a $C^2$-smooth map. In a local coordinate $(U,z=x+\ii y)$ of $M$, write $g=g(z)|\dd z|^2$ and $\Hopf(f)=\varphi(z)\dd z^2$. We recall some formulas for $\varphi$ which can be found in \cite{WanAu1994,SchoenYau1997}.
\begin{align*}
    \varphi(z) &= \sigma (f(z)) f_z\overline{f_{\bar{z}}} = \frac{1}{4}(\norm*{f_x}^2-\norm*{f_y}^2-2\ii \inner*{f_x}{f_y}).
\end{align*}

\begin{definition}\label{def:holomorphic-density}
    Let
    \[
        H=\norm*{\partial f}^2(z) := \frac{\sigma(f(z))}{g(z)} \abs *{\frac{\partial f}{\partial z}}^2(z),\qquad
        L=\norm*{\bar{\partial} f}^2(z) := \frac{\sigma(f(z))}{g(z)} \abs* {\frac{\partial f}{\partial \bar{z}}}^2(z).
    \]
\end{definition}

\begin{lemma}\label{lem:relation-density}
    The following equations hold:
    \begin{itemize}
        \item $e(f)=H+L$.
        \item $J(f)=H-L$.
        \item $\norm*{\Hopf(f)}^2=HL$.        
    \end{itemize}
    Here $J(f)$ is the Jacobian of $f$.
\end{lemma}

In later sections, we will consider the Dirichlet problems for maps from compact surfaces with smooth boundary into $\HH^2$. Since $\HH^2$ is complete and simply connected and has constant sectional curvature $-1$, \cite[Theorem 2.2]{KorevaarSchoen1993} implies that, for every compact surface $\Sigma$ with smooth boundary and every prescribed smooth boundary map $f:\partial \Sigma \to \HH^2$,  there exists a unique harmonic map $u:\Sigma \to \HH^2$ with trace $f$. Moreover, $u$ is the unique global minimizer of Dirichlet energy among all $W^{1,2}$ maps with the same trace. By \cite[Theorem 4 and the subsequent remark]{HildebrandtKaulWidman1977}, $u$ is smooth up to the boundary. This energy-minimizing property is important in the proof of the main theorem.

\subsection{Scherk maps}\label{subsec:scherk}

We recall the following two theorems of \cite{HanTamTreibergsWan1995}.

\begin{theorem}\label{thm:polygon-to-map}
    Suppose $k$ is a positive integer. For every ideal $(k+2)$-polygon $P$ in $\HH^2$, there is an orientation-preserving harmonic diffeomorphism $\omega\colon \C \to \HH^2$ such that $\omega(\C)$ is the interior of $P$ and $\Hopf(\omega)$ is a polynomial of degree $k$.
\end{theorem}

\begin{theorem}\label{thm:polynomial-to-map}
    For a nonconstant polynomial quadratic differential $\Phi$, there is an orientation-preserving harmonic diffeomorphism $\omega \colon \C \to \HH^2$ such that $\Hopf(\omega)=\Phi$ and the image of $\omega$ is the interior of an ideal polygon. Different choices of $\omega$ differ by postcomposition with an isometry of $\HH^2$. Following \cite{Huang2016Harmonic}, we call $\omega$ a
    \emph{Scherk map} associated with $\Phi$.
\end{theorem}

\begin{definition}\label{def:normalizedscherk}
    Suppose $\Phi$ is a nonconstant polynomial quadratic differential and $\omega$ is a Scherk map associated with $\Phi$. There is only one choice of $\omega$ such that 
    \begin{itemize}
        \item $\omega(0)=0$;
        \item $\dd \omega|_{0} \left((1,0)\right) = (a,0)$ for some $a>0$.
    \end{itemize}
    We call this $\omega$ the normalized Scherk map associated with $\Phi$.
\end{definition}

For later use, we separately discuss two classes of special cases: those in which the polygon $P$ possesses certain symmetries, and those in which the coefficients of $\Phi$ satisfy certain conditions.

\begin{proposition}\label{prop:reflection-symmetry-from-polygon}
    For an ideal polygon $P$ that is symmetric with respect to the real axis, we can choose the map $\omega$ in \cref{thm:polygon-to-map} such that $\omega(\bar{z}) = \overline{\omega(z)}$.
\end{proposition}

\begin{proof}
    The proof relies on \cite[Proposition 2.2, Theorem 2.3]{HanTamTreibergsWan1995}. 

    Let the ideal vertices of $P$ be 
    \[
        \xi_1,\dots,\xi_n \in \partial_\infty \HH^2.
    \]

    By the symmetry assumption, there exists a permutation $p$ on $\{1,2,...,n\}$ such that $\bar{\xi_i} = \xi_{p(i)}$.

    Let $M \subset \R^{2,1}$ be the entire spacelike surface of constant mean curvature $1$ whose Gauss image is the interior of $P$ and 
    \begin{equation*}
        \left\{
        \begin{aligned}
            &\Lambda_M(x)=0, \quad x=\xi_1,\xi_2,\dots \xi_n, \\
            &\Lambda_M(x)=-\infty,\quad x \in S^1 \backslash\{\xi_1,\xi_2,\dots,\xi_n\}
        \end{aligned}
        \right.
    \end{equation*}
    The definition of $\Lambda$ can be found in \cite[Section 2]{HanTamTreibergsWan1995}. Such $M$ exists uniquely by \cite[Proposition 2.2]{HanTamTreibergsWan1995}, it is the graph of a function $f: \C \to \R$ and it must be parabolic. In fact, $\Lambda$ is given by
    \[
        \Lambda_M(x) = \lim_{r\to +\infty} f(rx)-r
    \]
    for $x \in S^1$. We choose a holomorphic parameterization of $M$:
    \[
        F:\C \to M.
    \]
    Denote the Gauss map of $M$ by $G$.

    Let $L\left((x_1,x_2,x_3)\right)= (x_1,-x_2,x_3)$. Consider the surface $L(M)$. It satisfies the following properties:
    \begin{itemize}
        \item It is entire spacelike and has constant mean curvature $1$;
        \item Its Gauss image is the interior of $P$;
        \item $\Lambda_{L(M)}= \Lambda_{M}$.
    \end{itemize}
    The first two properties are obvious. The third one follows by the observation that $L(M(\mathbf{0}))$ is the graph of $g(x_1,x_2) = f(x_1,-x_2)$ and the symmetry assumption. In fact, by the definition of $\Lambda$,
    \begin{align*}
        \Lambda_{L(M)}(\xi_i) &= \lim_{r\to +\infty} g(r\xi_i)-r \\
        &= \lim_{r\to +\infty} f(r\bar{\xi}_i)-r \\
        &= \lim_{r\to +\infty} f(r\xi_{p(i)})-r  \\
        &=\Lambda_{M}(\xi_{p(i)}) = \Lambda_{M}(\xi_i)=0.
    \end{align*}
    A similar computation shows that $\Lambda_{L(M)}$ takes value $-\infty$ at $x \in S^1 \backslash \{\xi_1,\xi_2,\dots,\xi_n\}$. The uniqueness then gives $L(M) = M$. Thus $L$ can be viewed as an isometric involution of $M$ and the Gauss map of $M$ satisfies $G \circ L =\overline{G}$.
    
    Now consider $\tilde{L} = F^{-1} \circ L \circ F$. Since $L$ is an isometry and reverses the orientation, $\tilde{L}$ is an antiholomorphic involution of $\C$. Thus after a change of coordinate if necessary, we may assume that $\tilde{L}(z)=\bar{z}$. Let $\omega = G \circ F$. Then
    \[
        \omega(\bar{z}) = G \circ F \circ \tilde{L} (z) = G \circ L \circ F(z) =\overline{G(F(z))} = \overline{\omega(z)}.
    \]
\end{proof}

\begin{proposition}\label{prop:reflection-symmetry-from-polynomial}
    Suppose $\Phi$ is a nonconstant polynomial quadratic differential with real coefficients. Then the normalized Scherk map $\omega$ associated with $\Phi$ (see \cref{def:normalizedscherk}) satisfies
    \[
        \omega(\bar{z}) = \overline{\omega(z)}.
    \]
\end{proposition}

\begin{proof}
    For the normalized Scherk map $\omega$ associated with $\Phi$, let
    \[
        \tilde{\omega}(z) = \overline{\omega(\bar{z})}.
    \]
    It is straightforward to check that $\tilde{\omega}$ satisfies the same normalization conditions as $\omega$ and 
    \[
        \Hopf(\tilde{\omega})(z) =\overline{\Hopf(\omega)(\bar{z})}= \Hopf(\omega)(z).
    \]
    The last equality is because the coefficients of $\Phi$ are real. Thus $\tilde{\omega}=\omega$.
\end{proof}

In the subsequent sections, we are particularly interested in the case that the coefficients of $\Phi$ are real and neither $1$ nor $-1$ is a vertex of the image of the normalized Scherk map $\omega$ associated with $\Phi$. The following proposition provides a complete characterization of this case. Note that the order of $\Phi$ must be even.

\begin{proposition}\label{prop:criterion-for-pm1}
    For a non-constant polynomial quadratic differential $\Phi$ with real coefficients and even order, let $\omega$ be the normalized Scherk map associated with $\Phi$. Then neither $1$ nor $-1$ is a vertex of $\omega(\C)$ if and only if the leading coefficient of $\Phi$ is negative. 
\end{proposition}

\begin{proof}
    Suppose
    \[
        \Phi= p(z)\dd z^2 = \big(c_dz^d+\dots+c_0\big) \dd z^2,
    \]
    where all the coefficients are real and $d$ is even. Away from the zeroes of $\Phi$, there are precisely $d+2$ horizontal directions, and we can choose $d+2$ horizontal rays $h_1,\dots,h_{d+2}$ such that each of them associates to one direction and escapes to infinity. Along these $d+2$ rays, the map $\omega$ converges to $d+2$ distinct vertices of $\omega(\C)$, respectively. See \cite[Section 3]{HanTamTreibergsWan1995} for a proof.

    We now compute these directions. For a fixed $i\in \{1,2,\dots,d+2\}$ and the natural coordinate
    \[
        w(z) = \int ^z \sqrt{p(\zeta)}\dd \zeta,        
    \]
    we have
    \[
        w(h_i(t))=t + w_0
    \]
    for some constant $w_0$.

    Now write $N=(d+2)/2, A=\sqrt{c_d}/N$. Near infinity, we have
    \[
        w(z)=Az^N\big(1+o(1)\big),
    \]
    and thus
    \[
        w(h_i(t))=Ah_i(t)^N+o\big(|h_i(t)|^N\big)
    \]
    for all $t$ sufficiently large. Since $w(h_i)$ is horizontal, we must have
    \[
        \frac{Ah_i(t)^N}{\abs*{Ah_i(t)^N}} \rightarrow 1,\quad t\rightarrow +\infty.
    \]
    This is equivalent to
    \[
        \aarg(c_d)+(d+2)\aarg (h_i(t)) \rightarrow 0\quad  (\mathrm{mod}\,2\pi),  \quad t\rightarrow +\infty.
    \]
    We then conclude that
    \[
        \aarg(h_i(t)) \rightarrow \theta_k, \quad  t\rightarrow +\infty.
    \]
    for some $k$. Here
    \[
        \theta_k=\frac{-\aarg c_d+2k\pi}{d+2}, \quad k=0,1,\dots,d+1.
    \]
    The $\theta_k$ is called the horizontal asymptotic direction of $h_i$.

    Assume that 
    \[ 
        \omega (h_i(t)) \rightarrow v \in \partial_\infty \HH^2, \quad t\rightarrow +\infty.
    \] 
    Suppose
    \[
        w(t\e ^{\ii \theta_k})=X(t)+\ii Y(t),
    \]
    we have
    \[
        X(t)\rightarrow+\infty,\quad Y(t)=o\big(X(t)\big), \quad t\rightarrow +\infty.
    \]
    It is known that (see \cite[section 5]{Han1996Remarks}), in natural coordinate, the length element of a vertical line segment is
    \[
        \dd s = \sqrt{2\left(\cosh(2u)-1\right)} \dd y
    \]
    and the function $u$ obeys the estimate:
    \[
        u(z) \leq \frac{C}{\cosh(d_{\Phi}(z)/2\sqrt{2})}.
    \]
    Here $C>0$ is an absolute constant, $d_\Phi(z)$ is the $\Phi$-distance between $z$ and the zeroes of $\Phi$, which is larger than $a\re(w(z))$ for some constant $a>0$.

    For each $t$, denote $\zeta_t$ the point on $w(h_i)$ that has real part $X(t)$. We have
    \[
        d_{\HH^2}\Big(\omega\big(t\e^{\ii\theta_k}\big),\omega\big(w^{-1}(\zeta_t)\big)\Big) \leq \abs *{Y(t)-\im (w_0)} \e ^{-b X(t)}
    \]
    for some constant $b>0$ and $w_0 \in \C$. Therefore the distance tends to zero as $t$ tends to $+\infty$, and we establish
    \[
        \omega(t \e ^{\ii\theta_k}) \rightarrow v, \quad t \rightarrow +\infty.
    \]
    Denote the radial ray $\{t \e ^{\ii \theta_k}\}$ by $l_k$. Since the images of distinct $h_i$'s tend to distinct vertices, it follows that $\omega(l_k)$ and $\omega(l_j)$ must tend to different vertices if $k\neq j$.

    From the argument above, each ideal vertex $v$ of $\omega(\C)$ can be approached along exactly one of the $d+2$ radial rays $l_0,l_1,\dots,l_{d+1}$, i.e. for a certain $k=0,1,\dots,d+1$ determined by $v$,
    \[
        \omega(t \e^{\ii \theta_k}) \rightarrow v,\quad t \rightarrow +\infty.
    \]

    A curve $\gamma$ is horizontal if and only if its tangent vector $\gamma'$ satisfies
    \[
        \Phi(\gamma',\gamma')>0.
    \]
    If $\gamma$ is horizontal, define $\tilde{\gamma}=\bar{\gamma}$. The tangent vector of $\tilde{\gamma}$ satisfies
    \[
        \Phi(\tilde{\gamma}',\tilde{\gamma}')
        =\Phi(\overline{\gamma'},\overline{\gamma'})
        =\overline{\Phi(\gamma',\gamma')} 
        =\Phi(\gamma',\gamma')>0
    \]
    since the coefficients of $\Phi$ are real. Thus $\bar{\gamma}$ is still horizontal.
    
    The possible values of $\aarg(c_d)$ are $n\pi$, where $n \in \Z$. At least one of $\{1,-1\}$ is a vertex of $\omega(\C)$ if and only if $\omega(\C)$ has a fixed vertex by complex conjugation. This is equivalent to the existence of an index $k$ such that $\omega(l_k)$ and $\omega(\bar{l}_k)$ tend to the same vertex. Take $h_i$ which has horizontal asymptotic direction $\theta_k$. Since $\bar{h}_i$ is still a horizontal ray and $\omega(\bar{h}_i)$ tends to the same vertex as $\omega(h_i)$, we obtain $\bar{h}_i$ has the same horizontal asymptotic direction as $ h_i$. So we must have $\theta_k = n_k \pi$ for some integer $n_k$. This is possible if and only if $c_d>0$. 
\end{proof}

\begin{remark}\label{rem:normalization}
    For the special case $\Phi=a\dd z^2$ with $a \in \R^{-}$, we may choose $\omega$ such that 
    \begin{itemize}
        \item $\omega$ is independent of $x=\re(z)$;
        \item $\omega|_{\ii \R}$ gives a $2\sqrt{|a|}$-speed geodesic $l$ in $\HH^2$ which meet the real axis orthogonally at $\omega(0)$;
        \item $\omega(\bar{z}) =\overline{\omega(z)}$.
    \end{itemize}
    See \cref{rem:dz2}. We do not use the normalization in \cref{def:normalizedscherk}. The choice of the geodesic $l$ may vary depending on the situation. We regard this geodesic, together with its two ideal endpoints, as a degenerate ideal $2$-gon and continue to call these corresponding harmonic maps normalized Scherk maps. All the estimates below remain valid in this degenerate case.
\end{remark}

\subsection{Other auxiliary results}\label{subsec:auxiliary}
In this subsection, we collect several results from the literature that will be used in the subsequent sections.

The following Courant-Lebesgue Lemma allows us to deduce equicontinuity from a uniform energy bound. See \cite[Lemma 4-5]{Courant1937}.
\begin{theorem}[Courant-Lebesgue Lemma, \cite{Courant1937}]\label{thm:cl-lemma}
    Suppose $\Omega \subset \C$ is a planar domain and $K>0$ is a constant. For any $C^2$-smooth map $u:\Omega \to \HH^2$ with $E(u,\Omega) \leq K$, any $p\in \Omega$ and any $\delta \in (0,1)$ such that the closure of $B(p,\sqrt{\delta})$ is contained in $\Omega$, there exists $\rho \in (\delta ,\sqrt{\delta})$ so that for all $x,y \in  \partial B(p,\rho)= \{q: |p-q| = \rho \}$, we have
    \[
        d_{\HH^2}(u(x),u(y)) \leq \sqrt{\frac{8\pi K}{\abs *{\log(\delta)}}}.
    \]
\end{theorem}

To establish such energy bound, we need an interior gradient estimate \cite{Cheng1980Liouville} and the Isoenergy inequality \cite[Theorem 2.4]{Choe1998Isoenergy}:
\begin{theorem}[Cheng's interior gradient estimate, \cite{Cheng1980Liouville}]\label{thm:gradient}
    Let $M$ be a Riemannian manifold with Ricci curvature bounded from below by some non-positive constant $-K\leq 0$. Let $N$ be a complete simply connected Riemannian manifold with non-positive sectional curvature. Let $B(x_0,a)$ denote the closed ball of radius $a>0$ and centered at a fixed point $x_0$ in $M$. Let $f: M \to N$ be harmonic. Let $y_0$ be a fixed point lying outside of $f(B(x_0,a))$. Denote $r$ the distance from $x_0$ on $M$ and $\rho$ the distance from $y_0$ on $N$. Choose
    \[
        b>\sup_{B(x_0,a)} \rho\circ f(x)
    \]
    and let
    \[
        \beta = \inf_{B(x_0,a)} b^2- (\rho\circ f)^2.
    \]
    Then on $B(x_0,a)$ we have
    \[
        \frac{(a^2-r^2)^2\norm*{\nabla f}^2}{(b^2-(\rho \circ f)^2)^2} \leq c \max\left\{ \frac{Ka^4}{\beta}, \frac{a^2(1+Ka)}{\beta},\frac{a^2b^2}{\beta^2} \right\}.
    \]
    Here $c$ is a constant depending only on the dimension of $M$.
\end{theorem}

\begin{theorem}[Isoenergy inequality, \cite{Choe1998Isoenergy}]\label{thm:isoenergy}
    If $u$ is a smooth harmonic map from a closed unit Euclidean ball $B \subset \R^n,n\geq 2$ to $N$ of non-positive sectional curvature, then
    \[
        (n-1)E(u,B) \leq E(u|_{\partial B},\partial B).
    \]
\end{theorem}

In $\HH^2$, suppose $\gamma$ is a smooth curve with small total absolute geodesic curvature. Then the distance between the two end points of $\gamma$ has a lower bound given by a multiple of the length of $\gamma$.

\begin{lemma}\label{lem:arclength-to-distance}
    For $p,q \in \HH^2$, let $\gamma$ be a smooth curve connecting $p$ and $q$ with
    \[
        \int_\gamma |k_g| \leq A <1.
    \]
    Then 
    \[
        d_{\HH^2}(p,q) \geq l_{\HH^2}(\gamma) (1- A).
    \]
\end{lemma}

\begin{proof}
    Choose arclength parameterization of $\gamma$. Let $p=\gamma(0), T=\gamma'(0)$. Choose a geodesic ray $\alpha$ such that $\alpha(0)=p,\quad \alpha'(0)=-T$. Define the Busemann function
    \[
        B(x)=\lim_{t\to +\infty} d_{\HH^2}\big(x,\alpha(t)\big)-t.
    \]
    Denote $\zeta=\alpha(+\infty) \in \partial_{\infty}\HH^2$ the positive end point of $\alpha$. As observed in \cite[\S 2.6(a), p.~743]{BessonCourtoisGallot1995}, the vector $-\nabla B(x)$ is the unit tangent vector at $x$ pointing toward
    $\zeta$. In particular, $\norm*{\nabla B}=1$.

    By the computation of the Hessian of the Busemann function in \cite[\S 5(b), p.~751]{BessonCourtoisGallot1995}, one has
    \[
        \nabla^2 B = g-dB\otimes dB.
    \]
    Cauchy-Schwarz inequality yields that
    \[
        \nabla^2B(X,X)=\norm*{X}^2-\inner{X}{\nabla B}^2\geq 0.
    \]
    Set $f(s)=B(\gamma(s))$. We have $f'(s)= \inner{\nabla B}{\gamma'}$. In particular, $f'(0)= \inner{\nabla B}{T} = 1$. The last equality is obtained by differentiating $B(\alpha(s))=-s$.

    We then compute
    \[
        f''(s)=\nabla^2B(\gamma',\gamma') + \inner{\nabla B}{\nabla_{\gamma'} \gamma'}.
    \]
    Thus
    \[
        f''(s) \geq -|k_g(\gamma(s))|.
    \]
    Integrating both sides, we obtain
    \[
        f'(s) \geq 1- \int_0^s |k_g(\gamma(t))| \dd t.
    \]
    A second integration gives
    \begin{align*}
        f(l)-f(0) &\geq l-\int_0^l \int_0^s |k_g(\gamma(t))| \dd t \dd s  \\
        &= l-\int _0^l(l-t) |k_g(\gamma(t))| \dd t \\
        &\geq l(1-A).
    \end{align*}
    Since $B$ is $1$-Lipschitz,
    \[
        d_{\HH^2}(\gamma(0),\gamma(l)) \geq f(l) - f(0) \geq l(1-A).
    \]
\end{proof}

\section{Local equicontinuity}\label{sec:equicontinuity}
 In this section, we derive local equicontinuity from a uniform energy bound. 

 \begin{proposition}\label{prop:equi}
     Let $\Sigma$ be a Riemann surface and let $\{u_s:\Sigma \rightarrow \mathbb{H}^2\}$ be a family of harmonic maps. If the energies $E(u_s,\Sigma)$ are uniformly bounded in $s$, then $\{u_s\}$ is locally equicontinuous. As a consequence, either $\{u_s\}$ has a subsequence that converges locally uniformly to a harmonic map $u:\Sigma \rightarrow \mathbb{H}^2$, or it has a subsequence that diverges locally uniformly to some $z \in \partial_{\infty}\mathbb{H}^2$. 
 \end{proposition}

 \begin{proof}
    For any harmonic map $f: \Sigma \rightarrow \mathbb{H}^2$ and any fixed point $z_0 \in \mathbb{H}^2$, the function
    \[        
        \rho_{f,z_0}(z) = d_{\HH^2}\big(z_0,f(z)\big)
    \]
    is subharmonic. Suppose that the energy is bounded by a positive number $K$. For any $\varepsilon>0$, choose $0<\delta<1$ sufficiently small so that 
$\sqrt{\dfrac{8 \pi K}{|\log \delta|}} < \varepsilon.$ The following argument is performed in a conformal coordinate disk $U$. All domain distances and balls below are Euclidean in the coordinate. 

By \cref{thm:cl-lemma}, for a compact subset $V$ of the coordinate disk $U$, choose $\delta>0$ such that the closure of $B(x,\sqrt{\delta})$ is contained in $U$ for any $x \in V$. For any index $s$, any $x \in V$,  there exists $\xi_{x,s} \in (\delta,\sqrt{\delta})$ such that, for any $p,q \in L_{\xi_{x,s}}:=\{w \in U | d(w,x)=\xi_{x,s} \}$, the hyperbolic distance between $u_s(p)$ and $u_s(q)$ satisfies
    \[
        d_{\mathbb{H}^2}\big(u_s(p),u_s(q)\big) \leq \sqrt{\dfrac{8 \pi K}{|\log \delta|}} < \varepsilon.
    \]       
    Let $\overline{B(x,\xi_{x,s})}$ be the compact ball centered at $x$ with radius $\xi_{x,s}$. Choose $z_1,z_2 \in u_s\big(\overline{B(x,\xi_{x,s})}\big)$ such that
    \[
        d_{\mathbb{H}^2}(z_1,z_2)=\diam\, u_s\big(\overline{B(x,\xi_{x,s})}\big).
    \]
    By the subharmonicity of the function $\rho_{u_s,z_2}$ established above, we can choose the maximizing point $z_1 \in u_s(L_{\xi_{x,s}})$. The same holds for $z_2$.

    In particular, as long as $d(x,w)<\delta$, we find
    \[
        d_{\mathbb{H}^2}\big(u_s(x),u_s(w)\big)\leq \diam\, u_s\big(B(x,\xi_{x,s})\big) \leq \varepsilon.
    \]
    This shows the local equicontinuity of the family $\{u_s\}$.

    Suppose that there exist a subsequence $\{u_{s_j}\}$ and a point $x \in \Sigma$ such that $\{d_{\mathbb{H}^2}(u_{s_j}(x),0)\}$ is bounded in $j$. Let 
    \[
        x\in\Omega_1 \subset \Omega_2 \subset \dots
    \]
    be a compact exhaustion of $\Sigma$. For each fixed $i$, the equicontinuity of $\{u_{s_j}\}$ implies that $u_{s_j}(\Omega_i)$ is uniformly bounded in $j$. Therefore, the Arzelà--Ascoli theorem implies that $\{u_{s_j}|_{\Omega_i}\}_{j=1,2,...}$ is precompact.

    It remains to consider the case in which $d_{\mathbb{H}^2}(u_s(x),0) \rightarrow \infty$ for every $x \in \Sigma$. If we regard the $u_s$ as maps into $\mathbb{R}^2$, they are automatically uniformly bounded and equicontinuous with respect to the Euclidean metric. Thus, there is a subsequence $\{u_{s_j}\}$ that converges locally uniformly, in the Euclidean sense, to a map $u$. We have $u(\Sigma) \subset \partial _{\infty} \mathbb{H}^2$, and hyperbolic equicontinuity again implies that $u(\Sigma)$ consists of a single point. This completes the proof.
\end{proof}

\section{Distance between parachute maps and Scherk maps} \label{sec:boundedness}

Let $\Phi=\varphi(z)\dd z^2$ be a polynomial quadratic differential with $\deg \varphi =k-2 \geq 0$. Assume $\Phi$ has real coefficients, even degree and negative leading coefficient. Let $\omega:\mathbb{C} \rightarrow \mathbb{H}^2$ be the normalized Scherk map associated with $\Phi$. Let $\Omega_{1,s}$ be the annulus $\{1<|z|<s\} \subset \mathbb{C}$, and let $\gamma_r$ be the circle $\{|z|=r\}$.

Consider the following mixed Dirichlet–Neumann boundary value problem for harmonic maps:
\begin{equation*}
\left\{
\begin{aligned}
    &u^s:\Omega_{1,s} \rightarrow \mathbb{H}^2,\\
    &u^s|_{\gamma_s}=\omega|_{\gamma_s},\\
    &\dd u^s(\mathbf{n})|_{\gamma_1}=0.
\end{aligned}
\right.
\end{equation*}
Here $\mathbf{n}$ is the outer unit normal vector, that is, $\mathbf{n}=-\partial/\partial r$ on $\gamma_1$. The solution $u_s$ is called the parachute map. The name comes from \cite{Huang2016Harmonic}. This boundary-value problem has a unique solution. In fact, $u_s$ is the restriction of the unique solution to the following Dirichlet problem for harmonic maps:
\begin{equation*}
\left\{
\begin{aligned}
    &f^s:\Omega_{\frac{1}{s},s} \rightarrow \mathbb{H}^2,\\
    &f^s|_{\gamma_s}(z)=\omega|_{\gamma_s}(z),\\
    &f^s|_{\gamma_{{1}/{s}}}(z)=\omega|_{\gamma_{s}}({1}/{\bar{z}}).
\end{aligned}
\right.
\end{equation*}
In addition, $f^s$ is precisely the harmonic map obtained from $u^s$ by reflection across $\gamma_1$ via the Schwarz reflection principle. By a slight abuse of notation, we also write $u^s$ for $f^s$. We emphasize that $u^s$ minimizes the energy among maps in $\{ v \in W^{1,2}(\Omega_{1,s},\HH^2), v|_{\gamma_s}=\omega|_{\gamma_s} \}$. The reflection symmetry across $\R$ of $\omega$ induces the same symmetry of $u^s$, i.e. $\overline{u^s(z)}=u^s(\bar{z})$.

\begin{proposition}\label{prop:usbound}
    There is a constant $M>0$ such that for all $s$ and all $a\in (-s,-1/s)\cup(1/s,s)$, 
    \[
        d_{\HH^2}\big(0,u^s(a)\big) \leq M.
    \]
\end{proposition}

\begin{proof}
    The case $k=2$ is trivial by the normalization stated in \cref{rem:normalization}. Assume now $k>2$. Due to the reflection symmetry across $\R$, we have $u^s(a) \in \R$. The image of $u^s$ is contained in the polygon $\omega(\C)$. The length of the segment $\omega(\C) \cap (-1,1)$ is finite, since neither $1$ nor $-1$ is an ideal vertex of $\omega(\C)$. Recall that $0=\omega(0) \in \omega(\C) \cap (-1,1)$.
    Thus
    \[
        d_{\HH^2}\big(u^s(a),0\big) \leq l_{\HH^2} \big(\omega(\C) \cap (-1,1)\big) = M <+\infty.
    \]
\end{proof}

\begin{proposition}\label{prop:usomegabound}
    Let $f^s(z)=d_{\HH^2}\big(u^{s}(z),\omega(z)\big)$. Then there is a constant $B>0$ independent of $s$ such that $f^s\leq B$ on $\Omega_{1,s}$.
\end{proposition}

\begin{proof}
    Let $A=\{1<|z|<s, 0<\aarg(z)<\pi\}$ be the upper annulus. By \cref{prop:usbound} and the fact that $\omega(\R)$ is bounded, we have $f^s\leq M$ on $\partial A \cap \R$ for a constant $M>0$ independent of $s$. The boundary value gives $f^s=0$ on $\gamma_s=\{|z|=s\}$. Thus we need only to prove a uniform bound on $\gamma_1$.

    Assume $f^s\left(\e^{\ii\theta}\right)>0$ for some $\theta \in \R$. Let $l:[0,d] \to \HH^2$ be the unit speed geodesic connecting $u^s(\e^{\ii\theta})$ and $\omega\left(\e^{\ii\theta}\right)$. Let $z(t)= t \e^{\ii \theta}$. Then we have
    \begin{equation}\label{eq:fsnormal}
        \frac{\dd}{\dd t} f^s\big(z(t)\big)\Big|_{t=1} = -\inner{l'(0)}{\dd u^s(-\mathbf{n})} + \inner{l'(d)}{\dd \omega(-\mathbf{n})} = - \inner{l'(d)}{\dd \omega(\mathbf{n})}.
    \end{equation}
    Recall that $\mathbf{n}$ is the outer unit normal vector. Let $K$ be the maximum of $\norm*{\dd \omega(\mathbf{n})}$ on $\gamma_1$, then $\norm *{\dd f^s(\mathbf{n})}(\e^{\ii\theta}) \leq K$. 

    Let $g(z) = M+4K\re\left(\e^{\ii \pi/4}z^{-1/2}\right)$ be defined on a neighborhood of the closure of $A$. Here the branch of the square root is chosen so that it maps a positive real number to a positive real number. We have 
    \begin{equation}\label{eq:gnormal}
        \dd g(\mathbf{n})\left(\e ^{\ii\theta}\right)= 2K \cos \frac{\theta-\pi/2}{2} \geq \sqrt{2}K.
    \end{equation}
    Note that $g$ is positive on $\partial{A}$ and is no less than $M$ on $\partial A \cap \R$. Thus the subharmonic function $f^s-g$ is non-positive on $\partial A \cap \R$ and $\partial A \cap\gamma_s$. Assume that it attains a positive maximum at $\e^{\ii \theta_0}$. Let $z(t) = t \e ^{\ii \theta_0}$. We have $f^s(\e ^{\ii \theta_0})>0$ and
    \[
        \frac{\dd}{\dd t} \Big(f^s\big(z(t)\big)-g\big(z(t)\big)\Big)\Bigg|_{t=1} \leq 0.
    \]
    But by \eqref{eq:fsnormal} and \eqref{eq:gnormal}, the left side is
    \[
        - \inner{l'(d)}{\dd \omega(\mathbf{n})} + \dd g(\mathbf{n}) \geq (\sqrt{2}-1)K > 0,
    \]
    a contradiction. Thus $f^s \leq g \leq M+4K$ on $\partial A \cap\gamma_1$. This  together with the reflection symmetry of $u^s$ and $\omega$ across $\R$ give the desired bound $B=M+4K$.
\end{proof}

\begin{corollary}\label{coro:usboundoncircle}
    Fix $s_0>1$. There exists a constant $M>0$ such that 
    \[
        \sup_{ z \in \Omega_{1,s_0}} d_{\HH^2}\big(u^s(z),0\big) \leq M
    \]
    for all $s>2s_0$. Note that the inequality still holds if the domain of the supremum is replaced by $\Omega_{1/s_0,s_0}$.
\end{corollary}

\begin{proof}
    The triangle inequality
    \[
        d_{\HH^2}\big(u^s(z),0\big) \leq d_{\HH^2}\big(u^s(z),\omega(z)\big) + d_{\HH^2}\big(\omega(z),0\big)
    \]
    and \cref{prop:usomegabound} gives the desired result.
\end{proof}

We now show that the energy difference between the parachute maps and the Scherk map is bounded.

\begin{proposition}\label{prop:energydif}
    There is a constant $F>0$ such that for all $s>2s_0$,
    \[
        E(\omega,\Omega_{1,s})-F \leq E(u^s,\Omega_{1,s}) \leq  E(\omega,\Omega_{1,s}).
    \]
\end{proposition}

\begin{proof}
    The second inequality is immediate. Let $p_s$ be the unique solution to the following Dirichlet problem for harmonic maps:
    \begin{equation*}
        \left\{
        \begin{aligned}
            &p^s:B(0,1)\rightarrow \HH^2\\
            &p^s|_{\gamma_1}=u^s|_{\gamma_1}.
        \end{aligned} 
        \right.
    \end{equation*}
    Let $g^s$ be a continuous map from $\overline{B(0,1)} \cup \Omega_{1,s}$ to $\HH^2$ given by
    \begin{equation*}
        \left\{
        \begin{aligned}
            &g^s|_{B(0,1)}=p^s,\\
            &g^s|_{\Omega_{1,s}}=u^s.
    \end{aligned}
        \right.
    \end{equation*}
    The energy of $g^s$ is no less than that of $\omega$ on $\overline{B(0,1)} \cup \Omega_{1,s}$, that is,
    \[
        E\big(\omega,B(0,1)\big)+E(\omega,\Omega_{1,s}) \leq E\big(p^s,B(0,1)\big) + E(u^s,\Omega_{1,s}). 
    \]
    Thus it suffices to show that $E\big(p^s,B(0,1)\big)\leq F$ for all $s$.

    Now fix $0<a<1$ such that for all $z \in \gamma_1$, we have $B(z,a) \subset \Omega_{1/s_0,s_0}$. Let $q_0$ in $\HH^2$ be a fixed point with $d_{\HH^2}(q_0,0)=M+1$. Then $d_{\HH^2}(q_0,u^s(z)) \leq 2M+1$ for all $z \in \Omega_{1/s_0,s_0}$ and all $s$. Here $M$ is the constant stated in \cref{coro:usboundoncircle}.
    Let $b=2M+2$. For any $z_0 \in \gamma_1$ and $w \in B(z_0,a)$, We find $b^2- d_{\HH^2}^2(u^s(w),q_0) \geq 1$. By \cref{thm:gradient}, we establish
    \begin{equation}\label{eq:gradient}
        \norm*{\nabla u^s(z_0)}\leq c\frac{(2M+2)^6}{a^2},
    \end{equation}
    where $c$ is a constant independent of $s$ and $z_0$.
    
    Denote $\overline{\nabla}$ the gradient on the circle $\gamma_1$. The boundary value of $p^s$ gives
    \[
        \norm*{\overline{\nabla}p^s}=\norm*{\overline{\nabla}u^s} \leq \norm*{\nabla u^s} .
    \]
    Using \cref{thm:isoenergy}, we show
    \[
        E\big(p^s,B(0,1)\big) \leq E(p^s|_{\gamma_1},\gamma_1) = \frac{1}{2} \int_{\gamma_1} \norm*{\overline{\nabla}p^s}^2 \dd l \leq  \frac{1}{2}\int_{\gamma_1}\norm*{\nabla u^s}^2\dd l \leq F
    \]
    for some absolute constant $F$, as desired.
\end{proof}

\begin{remark} \label{rem:gradient}
    The inequality \eqref{eq:gradient} shows that the gradient of parachute maps are uniformly bounded in $s$ on $\gamma_1$, which will be used later. 
\end{remark}

\section{Proof of the main theorem}\label{sec:existence}

Suppose $\Sigma=X\backslash D$. Recall that $X$ carries an antiholomorphic involution $\iota$ with nonempty fixed-point set. We assume $D \subset \Fix(\iota)$. Let $\sigma=\sigma(z)\abs*{\dd z}^2$ be the hyperbolic metric on $\HH^2$.

We first prove a generalization of the Isoenergy inequality.

\begin{lemma}\label{lem:generalizediso}
    Suppose $S$ is a compact Riemann surface with boundary, and the boundary of $S$ is the union of $n$ analytic Jordan curves $\partial S = \cup_{i=1}^n l_i$.
    Let $f^s$ be a family of harmonic maps from $S$ to $\mathbb{H}^2$. If $f^s(\partial S)$ is uniformly bounded and $E(f^s|_{\partial S},\partial S)$ is also bounded, then $E(f^s,S)$ is bounded. Here and below, boundedness is uniform in $s$.
\end{lemma}

\begin{proof}
    By assumption, there are constants $M,K>0$ such that $d_{\HH^2}\big(f^s(\partial S),0\big) \leq M$ and $E(f^s|_{\partial S},\partial S) \leq K$ for all $s$. Around each $l_i$, choose a tubular neighborhood $U_i$ that is conformally equivalent, via a coordinate map $\xi_i$, to an annulus $\{\delta_i \leq |z| \leq 1\}$, with $\xi_i(l_i)$ equal to the unit circle. Denote by $q_i$ the other boundary component of $U_i$. We identify $U_i$ with $\{\delta_i \leq |z|\leq 1\}$. This may let the energy bound $K$ changed by multiplying a constant $c>0$.

    Write $b^s_i=f^s \circ \xi_i^{-1}|_{\gamma_1}$. The assumed hyperbolic distance bound implies
    \[
        \abs*{b_i^s(e^{\ii\theta})}\leq m:=\tanh\left(\frac{M}{2}\right)<1
    \]
    for every $i$, $s$, and $\theta$. Define
    \[
        \lambda_i(r)=\frac{r-\delta_i}{1-\delta_i},
        \qquad \delta_i\leq r\leq1.
    \]
    Let 
    \[
        \phi^s_i\left(r \e^{\ii \theta}\right)=\lambda_i(r) b^s_i\left(\e^{\ii \theta}\right).
    \]
    Its derivatives are
    \[
        \frac{\partial\phi_i^s}{\partial r} = \frac{1}{1-\delta_i}b_i^s\left(\e^{i\theta}\right)
    \]
    and
    \[
        \frac{\partial\phi_i^s}{\partial\theta} = \lambda_i(r) \frac{\partial b_i^s}{\partial\theta}\left(\e^{i\theta}\right).
    \]
    Therefore,
    \begin{align*}
        E(\phi_i^s,U_i) &= \frac{1}{2} \int_{\delta_i}^1\int_0^{2\pi} \left( \left|
        \frac{\partial\phi_i^s}{\partial r} \right|^2 +  \frac{1}{r^2} \left| \frac{\partial\phi_i^s}{\partial\theta} \right|^2 \right) \sigma\left(\phi_i^s\right)
        \,r\dd\theta\dd r \\
        &= I_{i,1}^s+I_{i,2}^s,
    \end{align*}
    where
    \[
        I_{i,1}^s = \frac{1}{2} \int_{\delta_i}^1\int_0^{2\pi} \frac{|b_i^s|^2}{(1-\delta_i)^2} \sigma\big(\lambda_i(r)b_i^s\big) \,r\dd\theta\dd r
    \]
    and
    \[
        I_{i,2}^s = \frac{1}{2} \int_{\delta_i}^1\int_0^{2\pi} \frac{\lambda_i(r)^2}{r^2} \left| \frac{\partial b_i^s}{\partial\theta} \right|^2 \sigma\big(\lambda_i(r)b_i^s\big) \,r\dd\theta\dd r.
    \]
    We first estimate the radial term. Since
    \[
        \abs*{\lambda_i(r)b_i^s}\leq \abs*{b_i^s}\leq m,
    \]
    we have
    \[
        \sigma\big(\lambda_i(r)b_i^s\big) = \frac{4}{\left(1-\lambda_i(r)^2|b_i^s|^2\right)^2} \leq \frac{4}{(1-m^2)^2}.
    \]
    It follows that
    \begin{align*}
        I_{i,1}^s &\leq \frac{1}{2} \int_{\delta_i}^1\int_0^{2\pi} \frac{4m^2}{(1-\delta_i)^2(1-m^2)^2} \,r\,d\theta\,dr \\
        &= \frac{2\pi m^2}{(1-m^2)^2} \frac{1+\delta_i}{1-\delta_i}.
    \end{align*}
    For the angular term, observe that $0\leq\lambda_i(r)\leq1$ and hence
    \[
        \sigma\big(\lambda_i(r)b_i^s\big) \leq \sigma\big(b_i^s\big).
    \]
    Thus, 
    \begin{align*}
        I_{i,2}^s &\leq \frac{1}{2} \int_{\delta_i}^1 \frac{\lambda_i(r)^2}{r}\,\dd r
        \int_0^{2\pi} \sigma\big(b_i^s\big) \left| \frac{\partial b_i^s}{\partial\theta} \right|^2 \,\dd\theta  \\
        &\leq cK \int_{\delta_i}^1 \frac{\lambda_i(r)^2}{r}\,\dd r  \\
        &\leq cK\log\frac{1}{\delta_i}.
    \end{align*}
    We obtain
    \[
        E\big(\phi_i^s,U_i\big) \leq \frac{2\pi m^2}{(1-m^2)^2} \frac{1+\delta_i}{1-\delta_i} + cK\log\frac{1}{\delta_i}.
    \]
    Now define a comparison map $g^s:S\to\HH^2$ by
    \[
        g^s(z) = 
        \begin{cases}
            \phi_i^s(z),\quad z\in U_i,\\
            0,\quad\mathrm{otherwise}.
        \end{cases}
    \]
    Since $\phi_i^s=0$ on the inner boundary of $U_i$, this defines a
    continuous $W^{1,2}$ map. Moreover, $f^s$ and $g^s$ have the same boundary value. 
    By the energy minimizing property, we have
    \begin{align*}
        E(f^s,S) &\leq E(g^s,S)\\
        &=\sum_{i=1}^n E\big(\phi_i^s,U_i\big)\\
        &\leq \sum_{i=1}^n \left( \frac{2\pi m^2}{(1-m^2)^2} \frac{1+\delta_i}{1-\delta_i} + cK\log\frac{1}{\delta_i} \right).
    \end{align*}
    This proves the lemma.
\end{proof}

Let $k\geq 1$ be an integer. For $i=1,2,\dots,k$, let $P_i(z)$ be an even order nonzero polynomial with real coefficients and the leading coefficient is negative. Let
\[
    \Phi_i(z) = P_i(z) \dd z^2   
\]
be a holomorphic quadratic differential. Denote by $\omega_i$ the normalized Scherk map associated with $\Phi_i$, and by $u^s_{i}$ the family of parachute maps corresponding to $\omega_i$. Recall that 
\[
    \overline{\omega_i(z)}=\omega_i(\bar{z})   
\]
and the ideal boundary of the image of $\omega_i$ contains neither $1$ nor $-1$, by \cref{prop:reflection-symmetry-from-polynomial} and \cref{prop:criterion-for-pm1}.  

For each $i$, the sets $u^s_i (\{1\leq |z| \leq r\})$ are uniformly bounded in $s>9r$ and $i$, due to \cref{coro:usboundoncircle}. 

Around each $p_i$, we choose $E_i$ to be the punctured neighborhood of $p_i$ that is invariant under $\iota$. Denote the coordinate map by $\eta_i : E_i \rightarrow\C$. We assume
$\eta_i(E_i)=\{|z|>1\}$ and $\eta_i \circ \iota \circ (\eta_i)^{-1} (z) = \bar{z}$.

Define
\[
    \Sigma_s= \left( \Sigma \backslash \bigcup \limits _{i=1}^k (E_i)\right) \bigcup \limits_{i=1}^k (\eta_i)^{-1}(\Omega_{1,s}).
\]
Here $\Omega_{1,s}=\{1<|z|<s\}$. This gives an exhaustion $\{\Sigma_s\}$ of $\Sigma$. Denote by $\gamma_{i,s}$ the boundary component of $\Sigma_s$ lying in $E_i$. We use the letter $\gamma$ because, in the coordinate chart, this is precisely the circle $\gamma_s$ defined in \cref{sec:boundedness}.

Consider the following Dirichlet boundary value problem for harmonic maps:
\begin{equation}\label{eq:dirichlet}
    \left\{
    \begin{aligned}
        &h^s:\Sigma_s \rightarrow \mathbb{H}^2,\\
        &h^s|_{\gamma_{i,s}}= \omega_i \circ \eta_i \big|_{\gamma_{i,s}}.
    \end{aligned}
    \right.
\end{equation}

We may drop the terms $\eta_i$ unless it is necessary to specify that the planar domain is in the chart of $E_i$ but not other $E_j$s. The boundary value condition then becomes
\[
    h^s|_{\gamma_{i,s}}= \omega_i|_{\gamma_{i,s}}.
\]

\begin{proposition}\label{prop:energybound}
    For fixed $r>1$, the energy $E(h^s,\Sigma_r)$ is uniformly bounded for $s>2r$. 
\end{proposition}

\begin{proof}
    Let $\tilde{\omega}_i(z)=\omega_i(rz) $ and $v^s_i$ the family of parachute maps correspond to $\tilde{\omega}_i$. The Hopf differential of $\tilde{\omega}_i$ is still a even order real polynomial with negative leading coefficient and thus $v^s_i(\{1 \leq |z| \leq r \})$ is uniformly bounded with respect to $s$, by \cref{coro:usboundoncircle}. 

    We now define a continuous map $g^s$ on $\Sigma_s$ with the same boundary values as $h^s$:
    \[
        g^s(z)=\begin{cases}
            v^{{s}/{r}}({z}/{r}), z \in E_i \cap (\Sigma_s\backslash \Sigma_r),\\
            y^s(z), z \in \Sigma_r.
        \end{cases}
    \]
    Here $y^s$ is the harmonic map from $\Sigma_r$ to $\mathbb{H}^2$ with the boundary value 
    \[
        y^s|_{\partial \Sigma_r \cap E_i} (z)= v^{{s}/{r}}_i({z}/{r}).
    \]
    This is indeed the value of $v^{{s}/{r}}$ on $\gamma_1$ composing a scaling. 

    As stated in \cref{rem:gradient}, the gradient $\nabla v^s_i$ is uniformly bounded on $\gamma_1$. Thus we have $E(y^s|_{\partial \Sigma_r \cap E_i},\partial \Sigma_r \cap E_i)$ is bounded. The boundedness of $\{v^s_i(\gamma_1)\}$ is established at the beginning of the proof. By \cref{lem:generalizediso}, we have $E(y^s,\Sigma_r)$ is bounded from above by a constant independent of $s$.

    Since $h^s$ minimizes the energy, we have
    \begin{align*}
        E\big(h^s,\Sigma_r\big) &\leq E\big(g^s,\Sigma_s\big)-\sum_{i=1}^k E\big(h^s, E_i \cap  \Sigma_s \backslash \Sigma_r  \big)\\
        &= E\big(y^s,\Sigma_r\big) +\sum_{i=1}^k E\left(v^{{s}/{r}}({\cdot}/{r}),E_i \cap \left( \Sigma_s \backslash \Sigma_r \right) \right)- E\big(h^s, E_i \cap  \Sigma_s \backslash \Sigma_r  \big)\\
        &\leq E\big(y^s,\Sigma_r\big).
    \end{align*}
    The last inequality follows from the energy-minimizing property of parachute maps stated in \cref{sec:boundedness}. This completes the proof.
\end{proof}

To prove that $\{h^s\}$ has a convergent subsequence, it suffices to show that $\{h^s(a)\}$ is bounded for some $a\in \Sigma$. This is established by the following proposition.

\begin{proposition}\label{prop:onepointbound}
    Fix $a\in \Fix(\iota)\backslash D$. There exists a constant $M>0$ such that, for all sufficiently large $s$,
    \[
        d_{\HH}\big(0,h^s(a)\big) \leq M.
    \]
\end{proposition}

\begin{proof}
    By taking $s$ sufficiently large, we may assume that $a \in \Sigma_s$ for every $s$. We have $\overline{h^s}=h^s\circ \iota$ by the same symmetry of its boundary value. Thus, $h^s(a) \in (-1,1)$ for all $s$. Let $U_i=\omega_i(\mathbb{C})$, and let $U$ be the convex hull of the union of $U_1,...,U_k$. The distance between points in $U \cap (-1,1)$ and $0$ is bounded by a constant $B>0$ because, by assumption, the ideal boundary of each $U_i$ contains neither $1$ nor $-1$. Therefore, $h^s(a) \in U\cap (-1,1)$ and
    \[
         d_{\HH^2}\big(0,h^s(a)\big) \leq B
    \]
    for all $s$.
\end{proof}

By \cref{prop:equi} and a standard diagonal argument, there exists a sequence $\{s_j\} \rightarrow +\infty$ such that $h^{s_j}$ converges locally uniformly on $\Sigma$ to a harmonic map $h$. 

\begin{proposition}\label{prop:asymptotic}
    For each $i$, the map $h$ is asymptotic to $\PP_i$ at the puncture $p_i$. Here $\PP_i$ is the ideal polygon that bounds $U_i=\omega_i(\C)$. 
\end{proposition}
\begin{proof}
    For the sake of simplicity, we drop the index $i$. Fix $r>1$. On the annulus region $ E  \cap (\Sigma_s\backslash\Sigma_r),$ the function
    \[
        f^s(z)=d_{\HH^2}\big(h^s(z),\omega(z)\big)
    \]
    is subharmonic. We have $f^s=0$ on $\gamma_s$. By the convergence of $\{h^{s_j}\},$ we have $f^{s_j}(z) \leq C$ on $\gamma_r$ for all $s_j.$ Here $C$ is a positive constant independent of $s$. Therefore 
    \begin{equation}\label{eq:bound-hj}
        ||f^{s_j}||_{L^\infty\left(E  \cap (\Sigma_{s_j}\backslash\Sigma_r)\right)} \leq C.
    \end{equation}
    Passing to the limit,
    \begin{equation}\label{eq:bound-h}
        d_{\mathbb{H}^2}\big(h(z),\omega(z)\big) \leq C
    \end{equation}
    for all $z \in E \backslash \Sigma_r$. Thus $h$ must reach exactly the same ideal boundary points as $\omega$ does. By \cref{prop:pole}, $\Hopf(h)$ is meromorphic at each of $p_1,\dots,p_k$. Therefore it must be asymptotic to $\PP_i$ at $p_i$.
\end{proof}

\begin{remark}\label{rem:absolute-constant}
    The constant $C$ comes from:
    \begin{itemize}
        \item the uniform bound for $\{ h^s(a)\}$,
        \item the equicontinuity of $\{h^s\}$.
    \end{itemize}
    The bound for $\{h^s(a)\}$ is determined as long as the surface $X$, the finite subset $D$ and all quadratic differentials $\Phi_1,\dots,\Phi_k$ are fixed. The contribution from the equicontinuity is controlled by the energy bound of $h^s$, which is given by the boundedness of parachute maps and their gradients. Thus $C_r$ only depends on $X,D,\Phi_1,\dots,\Phi_k$ and $r$. In particular, different choices of the convergent subsequence $\{h^{s_j}\}$ satisfy the same estimate 
    \eqref{eq:bound-hj}, and both of their limits, although may differ at present, satisfy \eqref{eq:bound-h}.
\end{remark}

\begin{proposition}\label{prop:unique}
    Given two harmonic maps $f, g: X\backslash D \to \HH^2$ such that
    \[
        \sup_{z\in X \backslash D} d_{\HH^2}\big(f(z),g(z)\big) \leq M < +\infty.
    \]
    Then $f$ must agree with $g$ identically, provided that either $f(X \backslash D)$ or $g(X \backslash D)$ is not contained in a geodesic in $\HH^2$.
\end{proposition}

\begin{proof}
    Let $r(x,y)=d_{\HH^2}(x,y)$ and $\rho=r\circ(f,g) \leq M$. Since $\rho$ is subharmonic and bounded on $X \backslash D$, it extends to each of the punctures. Thus we must have $\rho\equiv c \in [0,+\infty)$. We adapt the argument used in the proof of \cite[Lemma 3.11]{Sagman2023} to rule out the possibility of $c>0$.

    For each $p$, let $\gamma_p:[0,c] \to \HH^2$ be the unit speed geodesic segment $\gamma$ connecting $f(p)$ and $g(p)$. Let ${e_1,e_2}$ be an orthonormal basis of $T_p \Sigma$. Let $J_i$ be the Jacobi field along $\gamma$ with $J_i(0)=\dd f(e_i),J_i(c)=\dd g(e_i)$ for $i=1,2$. The second variation formula (see \cite[Section 2]{ChenJostWang2013}) gives:
    \[
        \nabla^2 r \Big( \big( \dd f(e_i),\dd g(e_i)\big),\big(\dd f(e_i),\dd g(e_i)\big) \Big) = \int _0 ^c \norm*{\nabla _{\gamma_p'} J_i^\perp(t)}^2 - R(J_i^\perp,\gamma_p',J_i^\perp,\gamma_p')\dd t.
    \]
    Summing over $i$:
    \begin{equation}\label{eq:variation}
        0= \Delta \rho = \sum_{i=1}^2 \int _0 ^c \norm*{ \nabla _{\gamma_p'} J_i^\perp(t)}^2 - R(J_i^\perp,\gamma_p',J_i^\perp,\gamma_p')\dd t.
    \end{equation}
    Decompose $J_i=a_i\gamma' + Y_i$, where $Y_i$ is orthogonal to $\gamma'$. Then \eqref{eq:variation} becomes:
    \[
        0=\sum_{i=1}^2\int_0^c \norm *{\nabla_{\gamma_p'}Y_i}^2+ \norm*{Y_i}^2 \dd t.
    \]
    Here we have used that the sectional curvature of $\HH^2$ is $-1$.
    Thus $Y_i$ vanishes along $\gamma$. 

    Since $\rho$ is constant, the first variation formula gives
    \[
        0=\dd\rho\big|_{p}(e_i)=\inner{J_i(c)}{\gamma_p'(c)}-\inner{J_i(0)}{\gamma_p'(0)} .
    \]
    Thus $a_i(c)=a_i(0)$. The Jacobi equation gives $a_i'' \equiv 0$. Therefore $a_i$ is a constant. We extend $\gamma_p$ to a complete unit speed geodesic with parameter $t\in\R$. Then the Jacobi field is still a constant proportion of the tangent field on the extended geodesic.

    Consider a smooth curve $l$ on $\Sigma$ parameterized by $s$, let $F(s,t)=\gamma_{l(s)}(t)$. From the discussion above, we conclude that
    \[
        F_*\frac{\partial}{\partial s} = a(s) F_*\frac{\partial}{\partial t}
    \]
    for a smooth function $a(s)$. For any fixed $(s_0,t_0)$, solve the following ODE:
    \begin{equation*}
        \left\{
        \begin{aligned}
            &\frac{\dd t}{\dd s}=-a(s),\\
            &\,t(s_0)=t_0.
        \end{aligned}
        \right.
    \end{equation*}
    Denote the solution by $\tau(s)$. We have
    \[
        \frac{\dd}{\dd s} F\big(s,\tau(s)\big) = 0.
    \]
    Thus 
    \[
        F(s_0,t_0)=F\big(0,\tau(0)\big).
    \]
    Define $\theta(s_0,t_0)=\tau(0)$. Since $(s_0,t_0)$ is arbitrary, we obtain
    \[
        F(s,t)= F\big(0,\theta(s,t)\big) \in \gamma_{0} 
    \]
    for all $(s,t)$. Thus the images of $f$ and $g$ must be contained in a geodesic in $\HH^2$, which is a contradiction.
\end{proof}

\begin{corollary}\label{coro:unique-limit}
    Suppose that the union of all $\PP_1,\dots,\PP_k$ is not contained in a geodesic. The family $\{h^s\}$ then converges to $h$ locally uniformly on $\Sigma$.
\end{corollary}
\begin{proof}
    Suppose $\{h^s\}$ has two limit points $h$ and $g$. By \cref{rem:absolute-constant}, we have $\rho=d_{\HH^2}(g,h) \leq C$ for an absolute constant $C>0$, independent of $h$ and $g$. Then $h=g$ by \cref{prop:unique}.
\end{proof}

\begin{proof}[Proof of \cref{thm:main}]
    On each member of the exhaustion $\{\Sigma_s\}$ of $\Sigma$, we solve the Dirichlet boundary value problem of harmonic maps \eqref{eq:dirichlet} and obtain a family of solutions $\{h^s\}$. The local equicontinuity of the family is given by \cref{prop:energybound}. This together with \cref{prop:onepointbound} show that the family $\{h^s\}$ is precompact. Any accumulation point of $\{h^s\}$ has a bounded distance away from the prescribed Scherk map near each of the punctures, by \cref{rem:absolute-constant}. The meromorphicity and the pole order of $\Hopf(h)$ at punctures is then given by \cref{prop:pole}. If the union of the images of all $\omega_i$ is not contained in a geodesic, \cref{prop:unique} then shows that the family only has one accumulation point $h$, and any harmonic map $g$ satisfies $d_{\HH^2}(g,\omega_i) \leq C$ near $p_i$ for each $i$ must coincide with $h$. This completes the proof of \cref{thm:main}.
\end{proof}

\section{A principal part criterion for bounded distance}\label{sec:principal-bound}

Recall that we have chosen the local coordinate $(E_i, \eta_i)$ near each puncture $p_i$. For any holomorphic quadratic differential $q$ on $X \backslash D$ which has a pole at each $p_i$ with order at least $3$, we define its principal part in this coordinate. The following definition is similar to \cite[Definition 2.5]{Gupta2021Harmonic}. The only difference is that the coordinate is parameterized by $\{|z|>1\}$ in our case, whereas by $\{0<|z|<1\}$ in Gupta's definition. 

\begin{definition}\label{def:principal}
    Let $q$ be as above. Its principal part at $p_i$, denoted by $\pr(q,p_i)$, is a truncation of the Laurent expansion of $\sqrt{q}$ to the $z^{-1}$ term near $p_i$ in the coordinate $(E_i,\eta_i)$. If the order of $q$ at $p_i$ is odd, we define the square root and the principal part on a double cover of $E_i$.

    In particular, if the order of $q$ at $p_i$ is an even number $2d+4$, its principal part has the form
    \[
        \pr(q,p_i)= \big(c_d z^d+\dots+c_0+c_{-1}z^{-1}\big) \dd z.
    \]
    Note that there is a sign flexibility of the choice of square root, as well as principal part. In the following, \textbf{two principal parts are equal} means that they are equal for suitably chosen signs.
\end{definition}

Suppose $\Hopf(f)$ is the Hopf differential of a harmonic map $f: X\backslash D \to \HH^2$ that is meromorphic at the punctures. We make further assumption that in coordinate $(E_i,\eta_i)$, it has the form
\begin{equation}\label{eq:assumptionhopf}
    \Hopf(f)= A_i(f)z^{N_i}\big(1+o(1)\big) \dd z^2
\end{equation}
for $N_i \geq 0$ as $z$ approaches $\infty$. Here $A_i(f)$ is a constant depends only on $f$ and $i$. By a computation similar to that in the proof of \cref{prop:criterion-for-pm1}, there are $N_i+2$ horizontal asymptotic directions
\[
    \theta_{i,k}= \frac{-\aarg A_i+2k\pi}{N_i+2}, \quad k=0,1,\dots,N_i+1.
\]

\begin{definition}\label{def:normalizedasymptotic}
    For an ideal polygon $P$ of $\HH^2$, suppose its ideal vertices are $\{\xi_1,\dots, \xi_{N_i+2}\}$ in cyclic order. We say $f$ is asymptotic to $P$ at $p_i$ in the normalized sense, if the image of $f$ approaches $\xi_{k+1}$ along the ray $\{t\e ^{\ii \theta_{i,k}}\}$ as $t$ tends to infinity in coordinate $(E_i,\eta_i)$ for each $k$. If we say $f$ is asymptotic to $P$ at $p_i$ in the normalized sense, we require that $\Hopf(f)$ satisfies \eqref{eq:assumptionhopf} priorly.
\end{definition}

It is shown in \cref{app:decayestimate} that, in a horizontal half plane $H$ of $q=\Hopf(f)$, the pullback of the hyperbolic metric is
\[
    f^*\sigma = 4 \cosh ^2\eta\dd x ^2 + 4 \sinh ^2\eta \dd y^2
\]
with 
\[
    |\eta|+|\nabla\eta| \leq A\e^{-cR}.
\]
Here $\eta=\frac{1}{2}\log H$ and $R$ is the distance to the boundary of the half plane. See \cref{lem:decay}. The constant $A,c>0$ depends only on $f$. As a consequence, the geodesic curvature of a horizontal line in the half plane satisfies
\[
    |\kappa_g| \leq \frac{1}{2}A \e^{-cR}.
\]
See \eqref{eq:kg-bound}.

For $f$ that is asymptotic to an ideal polygon $P$ in the normalized sense, we can associate the horizontal half plane $H$ with a unit speed geodesic edge $\gamma$ of $P$, such that the image of $f$ tends to $\gamma$ as the $y$-factor tends to infinity in $H$. There is only one direction for $y$ to go to infinity, depending on the choice of the sign of $\sqrt{q}$. Also note that, if we fix $y$ and let $x \to +\infty(-\infty,\text{respectively})$, the image of $f$ will tend to the positive ideal endpoint of $\gamma$ (the negative ideal end point of $\gamma$, respectively).

\begin{lemma}\label{lem:parameterofdepth}
    In a horizontal half plane $H$ of $f$, we have
    \[
        d_{\HH}\big( f(x,y), \gamma(2x+r) \big) \leq M
    \]
    for arbitrary $(x,y) \in H$ with $|y|>2$ and two constants $M>0,r\in \R$. 
\end{lemma}

\begin{proof}
    We have
    \[
        \big| \norm *{f_x}-2 \big| \leq C\e^{-c|y|}
    \]
    and
    \[
        \norm*{f_y} \leq C\e ^{-c|y|}.
    \]
    Fix $x$ and let $y \to \infty$, the limit
    \[
        u(x)= \lim_{y \to \infty} f(x,y)
    \]
    exists and lies in $\gamma$. Suppose $u(x)=\gamma(t_x)$. For any $x_1<x_2 \in R$, we have the estimate
    \[
        d_{\HH^2}(f(x_1,y),f(x_2,y)) \leq \int_{x_1}^{x_2} \norm*{f_x(t,y)} \dd t \leq 2|x_2-x_1| + C|x_2-x_1|\e^{-c|y|}.
    \]
    By \cref{lem:arclength-to-distance}, we also obtain the converse estimate
    \begin{align*}
        d_{\HH^2}(f(x_1,y),f(x_2,y)) &\geq \left( 1-\int_{x_1}^{x_2} |\kappa_g|\norm*{f_x} \dd t \right) \int_{x_1}^{x_2} \norm*{f_x(t,y)} \dd t \\
        &\geq  \left( 1- C'|x_2-x_1|\e^{-c|y|} \right)\left(2|x_2-x_1|-C'|x_2-x_1|\e^{-c|y|}\right).
    \end{align*}
    Passing to the limit, we have
    \[
        d_{\HH^2} \big(u(x_1),u(x_2)\big) = 2|x_1-x_2| = d_{\HH^2} \big(\gamma(t_{x_1}),\gamma(t_{x_2})\big).
    \]
    Thus $t_{x_1}-t_{x_2}=2(x_1-x_2)$, which is equivalent to $t_x=2x+r$ for a constant $r \in \R$. The distance satisfies
    \[
        d_{\HH^2}\big(f(x,y),\gamma(2x+r)\big) \leq \abs*{ \int_y^{\infty} CA\e^{-c|t|} \dd t} \leq M
    \]
    as desired.
\end{proof}

We now give a criterion for bounded distance between harmonic maps from $X \backslash D$ to $\HH^2$, using their principal parts. The proof of the following two propositions are analogous to that of \cite[Proposition 3.9]{Gupta2021Harmonic}.

\begin{proposition}\label{prop:principaltobounded}
    Suppose $f,g$ are two harmonic maps defined on a small neighborhood of $p_i$ to $\HH^2$ such that they are asymptotic to the same ideal polygon in the normalized sense and their Hopf differentials have the same principal part at $p_i$. Then the distance between $f$ and $g$ is bounded near $p_i$.
\end{proposition}

We refer a complete geodesic in $\HH^2$ as a degenerate $2$-gon. The proposition is also valid in this case. 

\begin{proof}
    We assume both pole orders are even, otherwise we work on a double cover. Without loss of generality, we may assume the neighborhood is $E_i$. We apply the coordinate map $\eta_i$. Recall that $\eta_i(E_i)=\{|z|>1\}$. Suppose all zeroes of $\Hopf(f)$ and $\Hopf(g)$ are in $\{1<|z|<R\}$. We choose a polygonal exhaustion of $\{|z|>1\}$ with respect to $\Hopf(f)$. That is, a sequence of enlarging "polygon"s $\Gamma(T_j)$ outside $\{|z|=R\}$ whose edges are alternating $\Hopf(f)$-horizontal and $\Hopf(f)$-vertical segments with $\Hopf(f)$-length $2T_j+\text{error}$. Here the error terms are bounded by a universal constant that is independent of $T_j$. See \eqref{eq:horizontalO1} and \eqref{eq:verticalO1}. "Enlarging" means that $\Gamma(T_{j+1})$ lies outside $\Gamma(T_j)$ and the sequence $T_j$ tends to infinity increasingly. The $\Hopf(f)$-distance between $\Gamma(T_j)$ and any point $z_0 \in \{1<|z|<R\}$ satisfies
    \[
        T_j \leq d_{\Hopf(f)}\big(z_0,\Gamma(T_j)\big) \leq D_1T_j+D_2
    \]
    for absolute constants $D_1,D_2>0$ that is independent of $z_0$ and $T_j$. For the explicit construction, see \cref{app:polygonexhaustion}.

    The distance function between two harmonic maps are subharmonic. Once we show that there is a constant $M>0$ independent of $T_j$ such that
    \[
        \sup_{z\in\Gamma(T_j)}  d_{\HH^2}\big(f(z),g(z)\big) \leq M
    \]
    for all $j$, the distance between $f$ and $g$ has an upper bound on the region bounded by $\{|z|=R\}$ and $\Gamma(T_j)$ for all $j$, since the distance between $f$ and $g$ on the inner boundary $\{|z|=R\}$ is automatically bounded. The proof is then finished.

    We order the edges of $\Gamma(T_j)$ cyclically by $\alpha_1(T_j), \beta_1(T_j),\dots \alpha_m(T_j), \beta_m(T_j)$. Each $\alpha_i(T_j)$ is horizontal and each $\beta_i(T_j)$ is vertical. 
    
    We first show that the distance between $f$ and $g$ is bounded on $\alpha_k(T_j)$. It is contained in the $\Hopf(f)$-horizontal half plane $H_k$. Denote $l$ the $\Hopf(f)$-straight line
    \[
        l:(-\infty,+\infty) \to H_k,\quad l(t)= t+\ii (b_k+\varepsilon_k T_j).
    \]
    Then $\alpha_k(T_j)$ is the restriction of $l$ to a closed interval. The sign $\varepsilon_k \in \{1,-1\}$ describes the direction that $y$ goes to infinity in $H_k$. See \cref{app:polygonexhaustion}. Suppose $\gamma_k$ is the geodesic associated with $H_k$. Let $\xi_k$ and $\xi_{k-1}$ be the positive ideal endpoint and the negative ideal endpoint of $\gamma_k$ respectively. 
    
    Although $l$ is not $\Hopf(g)$-horizontal, the assumption that the principal parts of $f$ and $g$ are equal implies that $l$ is asymptotic to a $\Hopf(g)$-horizontal direction. Since $f$ and $g$ are asymptotic to the same ideal polygon in the normalized sense, points in $\alpha_k(T_j)$ are contained in a $\Hopf(g)$-horizontal half plane that is also associated with $\gamma_k$. We denote this $\Hopf(g)$-horizontal half plane by $\tilde{H}_k$. This is because the assumption on principal parts forces the difference of natural coordinate function with suitable sign is bounded. For $z \in \alpha_k(T_j)$, suppose it is represented by $(x^f,y^f)$ in the natural coordinate of $\Hopf(f)$, then \cref{lem:parameterofdepth} gives
    \[
        d_{\HH^2}\left(f(z),\gamma\left(2x^f+r_k^f\right)\right) \leq M^f_k.
    \]
    Similarly, let $(x^g,y^g)$ be the $\Hopf(g)$-natural coordinate representation of $z$. Applying the same lemma to $g$ and $\tilde{H}_k$, we have
    \[
        d_{\HH^2}\big(g(z),\gamma\left(2x^g+r^g_k\right)\big) \leq M^g_k.
    \]
    According to \cite[Lemma 3.7]{Gupta2021Harmonic}, the difference $x^f-x^g$ is bounded by an absolute constant $C>0$. Thus
    \[
        d_{\HH^2} \big(f(z),g(z)\big) \leq M^f_k+M^g_k+2C+\abs*{r^f_k-r^g_k}
    \]
    for all $z \in \alpha_k(T_j)$.
    
    It remains to deal with the vertical segments. Let $A_k$ be the common endpoint of $\alpha_k(T_j)$ and $\beta_k(T_j)$. We need only to show that $d_{\HH^2} (f(z),f(A_k)) \leq M$ and $d_{\HH^2}(g(z),g(A_k)) \leq M$ for all $z \in \beta_k(T_j)$ and a constant $M$. Like before, $\beta_k(T_j)$ is contained in the $\Hopf(f)$-vertical half plane $V_k$. Denote $l$ the $\Hopf(f)$-straight line
    \[
        l:(-\infty,+\infty) \to V_k,\quad l(t)= a_k+\tau_kT_j +\ii t.
    \]
    Then $\beta_k(T_j)$ is the restriction of $l$ to a closed interval. The sign $\tau_k \in \{1,-1\}$ describes the direction that $x$ goes to infinity in $V_k$. See also \cref{app:polygonexhaustion}. Suppose the $A_k$ is represented by $(x^f_0,y^f_0)$ and $z$ is represented by $(x^f,y^f)$ in $V_k$. Since both $A_k$ and $z$ are in the $\Hopf(f)$-vertical straight line $l$,
    \[
        x^f_0=x^f=a_k + \tau_kT_j.
    \]
    We then obtain
    \[
        d_{\HH^2} (f(z),f(A_k)) \leq \abs*{\int_{y^f}^{y^f_0} \norm*{f_y} \dd t } \leq 2(T_j+\text{error})C_f\e^{-c_fT_j},
    \]
    the right side has a uniform upper bound $M$ independent of $j$.

    Let $(x_0^g,y_0^g)$ be the $\Hopf(g)$-natural coordinate representation of $A_k$ and $(x^g,y^g)$ be that of $z$. Let $\hat{z}$ be the point correspond to $(x^g_0,y^g)$. Again by \cite[Lemma 3.7]{Gupta2021Harmonic}, the $y$-factors satisfy 
    \[
        |y^g-y^g_0| \leq 2T_j +\text{error} +C,
    \]
    while the $x$-factors satisfy
    \[
        |x^g-x^g_0| \leq C
    \]
    and
    \[
        |x^g_0 - (a_k+\tau_kT_j)| \leq C.
    \]
    By triangle inequality,
    \begin{align*}
        d_{\HH^2}\big(g(z),g(A_k)\big) &\leq d_{\HH^2}\big(g(z),g(\hat{z})\big) + d_{\HH^2}\big(g(\hat{z}),g(A_k)\big) \\
        &\leq C\left(2+C_g\e^{-c_gT_j}\right) + 2(T_j+\text{error}+C)C_g\e^{-c_gT_j},
    \end{align*}
    the right side also has a uniformly upper bound independent of $j$. This completes the proof.
\end{proof}

\begin{proposition}\label{prop:boundtoprincipal}
    Let $f$ and $g$ be two harmonic maps defined near $p_i$ such that $\Hopf(f)$ and $\Hopf(g)$ are meromorphic at $p_i$ with pole order at least $4$. If the distance between $f$ and $g$ are bounded near $p_i$, then the principal part of $f$ and $g$ at $p_i$ must agree. 
\end{proposition}

\begin{proof}
    We assume only even pole order appears, otherwise we consider a double cover. 
    
    Choose a narrow angular sector neighborhood $S$ of $\infty$ such that as $z\in S$ approaches $\infty$, both $\re\left(\int^z \sqrt{\Hopf(f)}\right)$ and $\re\left(\int^z \sqrt{\Hopf(g)}\right)$ tend to $+\infty$, with a suitably chosen sign of square roots. 
    
    Choose a $\Hopf(f)$-horizontal half plane $H^f$ that contains $S$. Shrinking $S$ if necessary. For $z \in H^f$, let $(x^f(z),y^f(z))$ be its $\Hopf(f)$-natural coordinate representation. Let $\gamma$ be the geodesic associated to $H^f$. Although we have not made the assumption that $f$ is asymptotic to a polygon, the same computation as in \cref{lem:parameterofdepth} shows that this $\gamma$ exists, and the distance estimate in \cref{lem:parameterofdepth} is still valid. Similarly, we assign $\gamma'$ for a $\Hopf(g)$-horizontal half plane $H^g$ that contains $S$.
    
    Using the assumption that the distance between $f$ and $g$ is bounded near $p_i$, we obtain that $\gamma$ and $\gamma'$ has a same endpoint, denoted by $\xi$. We have
    \[
        d_{\HH^2} \left(\gamma\left(2x^f(z)+r^f\right), \gamma'\left(2x^g(z)+r^g\right) \right) \leq M'
    \]
    for some absolute constant $M'>0$. This implies that the difference $x^f-x^g$ is bounded. As a consequence, the coefficient of any nonnegative terms of the principal parts of $\Hopf(f)$ and $\Hopf(g)$ must be equal, and the coefficients of $z^{-1}$ term can only differ by a pure imaginary number. In particular, the pole order of $\Hopf(f)$ and $\Hopf(g)$, as well as their leading coefficients, are the same. 
    
    Thus the two differentials have the same horizontal asymptotic directions, and the limit geodesic for each horizontal half plane of $f$ and $g$ agree in a cyclic order. These geodesics form a twisted polygon $\PP$. By the same argument in \cite[Proposition 2.29]{Gupta2021Harmonic}, the metric residue of $\PP$ is twice the real part of the analytic residue of either principal parts, with compatible sign conventions. We remind that if the coefficient of $z^{-1}$ term is $c$, the analytic residue is defined to be $2\pi\ii c$ in Gupta's convention. This implies the imaginary part of the coefficient of $z^{-1}$ term of these two principal parts is equal. The proof is finished.
\end{proof}

\section{Comparison with horizontal catenoids}\label{sec:horizontalcatenoids}
Recall that, horizontal catenoids in $\HH^2 \times \R$ are complete embedding minimal annuli 
\[
    \Psi=(F,h): \C^* \to \HH^2 \times \R
\]
with finite total curvature. They are constructed by Pyo \cite{Pyo2011NewComplete} and Morabito–Rodríguez \cite{MorabitoRodriguez2012} independently. As stated in \cite[Section 2]{MartinMazzeoRodriguez2014}, a horizontal catenoid in $\HH^2 \times \R$ is isometric to a canonical one, denoted by $\Sigma$, that is symmetric across each of the three planes:
\[
    \{u=0\},\quad\{v=0\} ,\ \quad \{t=0\}. 
\]
Here $t$ is the parameter of $\R$, and $w=u+\ii v$ is the parameter of the unit disk model of $\HH^2$. Conversely, a Schoen-type theorem \cite[Main Theorem]{HauswirthNelliSaEarpToubiana2015} characterizes all horizontal catenoids in $\HH^2 \times \R$ as complete connected minimal surfaces in $\HH^2 \times \R$ with finite total curvature and two ends such that each of the ends is asymptotic to a vertical plane. We can choose a suitable coordinate such that
\begin{equation*}
    \left\{
        \begin{aligned}
            &F(\bar{z})=\overline{F(z)} \\
            &F(-\bar{z})=F(z) \\
            &F(1/\bar{z})=-\overline{F(z)}
        \end{aligned}
    \right.
\end{equation*}
and
\begin{equation*}
    \left\{
        \begin{aligned}
            &h(\bar{z})=h(z) \\
            &h(-\bar{z})=-h(z) \\
            &h(1/\bar{z})=h(z).
        \end{aligned}
    \right.
\end{equation*}
The surface $\Sigma$ has two ends, each of them is asymptotic to a vertical plane. Both ends have degree $0$ by \cite[Lemma 3.11; see also p. 231]{HauswirthNelliSaEarpToubiana2015}.

\begin{proposition}\label{prop:hopfofcatenoid}
    The Hopf differential of $F$ has the form
    \[
        \Hopf(F)= -a^2\left(1-\frac{1}{z^2}\right)^2\dd z^2
    \]
    for some real number $a>0$.
\end{proposition}

\begin{proof}
    By \cite[Eq. (2)]{HauswirthNelliSaEarpToubiana2015}, the square root of $\Hopf(F)$ has a pole of order $2$ at infinity. Thus the differential
    \[
        h_z\dd z=\ii \sqrt{\Hopf(F)}
    \]
    has a pole of order $2$ at infinity. Similarly, it has a pole of order $2$ at $0$. Thus we may write
    \[
        h_z\dd z= \left(a+\frac{b}{z}+\frac{c}{z^2}\right) \dd z.
    \]
    The symmetries of $h$ imply:
    \begin{itemize}
        \item All of $a$, $b$ and $c$ are real;
        \item $b=0$;
        \item $a=-c$.
    \end{itemize}
    Thus
    \[
        \Hopf(F)= - h_z^2\dd z^2= -a^2\left(1-\frac{1}{z^2}\right)^2\dd z^2.
    \]
\end{proof}

Note that $h$ is given by
\[
    h(z) = 2a\re\left(z+\frac{1}{z}\right).
\]
Thus the parameter $a$ is exactly $1/8$ times the distance between $\Psi(1)$ and $\Psi(-1)$. Moreover, $a$ is determined by the pullback metric $\Psi^*(\sigma+\dd t^2)$.

\begin{proposition}\label{prop:intrinsic}
    $a$ is determined by the pullback metric $\Psi^*(\sigma+\dd t^2)$.
\end{proposition}

\begin{proof}
    Denote $g=\Psi^*(\sigma+\dd t^2)$. By Gauss equation, the Gauss curvature of the metric $g$ at $1$ satisfies
    \[
        K(1)= -1 + \det(S), 
    \]
    because the tangent plane of $\Psi(1)$ is horizontal and thus has sectional curvature -1. Here $S$ is the shape operator.

    Decompose 
    \begin{equation}\label{eq:tangent-normal}
        \frac{\partial}{\partial t} = \Psi_*(T) + \nu N
    \end{equation}
    into a tangent part $\Psi_*(T)$ and a normal part $\nu N$. Here $N$ is the unit normal vector. We have $\norm*{T}^2 + \nu^2=1$. Take inner product with $\Psi_*(T')$, we obtain
    \[
        T'(h) = \inner{T}{T'}.
    \]
    Thus $T= \nabla h$.
    Take covariant derivative in $\HH^2 \times \R$ of both sides of \eqref{eq:tangent-normal}. Comparing the tangent part, we obtain
    \[
        \nabla _{T'} \nabla h = \nu S(T').
    \]
    Here we have used the Weingarten formula and the fact that $\partial/\partial t$ is parallel in $\HH^2 \times \R$. Note that $\nabla h$ vanishes at $1$, therefore $\nu(1)=1$ and 
    \[
        \norm*{\nabla^2 h}^2 (1) = \norm *{S}^2 (1).
    \]
    Since $\Psi$ is minimal, $S$ is traceless and thus $\det(S)=-\frac{1}{2} \norm *{S}^2$. We then have
    \[
        K(1)=-1-\frac{1}{2}\norm*{\nabla ^2 h}^2= -1 - 2a^2\norm*{\nabla ^2 f}^2, 
    \]
    where 
    \[
        f(z) = \re\left(z+\frac{1}{z}\right)
    \]
    is a fixed function. Since $\dd f$ is $0$ at $1$, the Hessian of $f$ at $1$ agrees with the usual Hessian in $\R^2$, which is not $0$ at $1$. The proof is then finished.
\end{proof}

Let $(E,\Phi)$ denote the $SL(2,\R)$-Higgs bundle associated with the harmonic map $F$. The Picard group of $\C^*$ is trivial, thus $E$ is the trivial vector bundle of rank $2$. The Higgs field $\Phi$ has the form
\[
    \Phi = \begin{pmatrix}
            0 & \alpha \\
            \beta & 0
            \end{pmatrix} \dd z.
\]
Here $\alpha$ and $\beta$ are holomorphic, with
\[
    \Hopf(F)=4\alpha\beta \dd z^2
\]
and
\[
    J_F=4(\norm*{\alpha}^2 - \norm*{\beta}^2).
\]

\begin{proposition}\label{prop:higgsfield}
    The Higgs field $\Phi$ has the form
    \[
    \Phi = \begin{pmatrix}
            0 & u\left(1+\frac{1}{z}\right)^2 \\
            v\left(1-\frac{1}{z}\right)^2 & 0
            \end{pmatrix} \dd z,
    \]
    or the transposition of that matrix. Here $u,v$ are holomorphic, they never vanish on $\C^*$ and satisfy $4uv=-a^2$. Note that choose different $u,v$ and transpose the matrix do not change the pullback metric.
\end{proposition}

\begin{proof}
    Since $\Psi$ is a conformal minimal embedding and $\Hopf(F)$ vanishes at $\pm1$, the Jacobian of $F$ satisfies $J_F(\pm1) \neq 0$. The symmetry $F(-\bar{z})=F(z)$
    forces $J_F(1)=-J_F(-1)$. So it is impossible to have
    \[
        \Phi = \begin{pmatrix}
                0 & q \\
                1 & 0
                \end{pmatrix}\dd z.
    \]
    Meanwhile, $\Phi$ cannot be a zero matrix at any point. Thus it must have the form
    \[
        \Phi = \begin{pmatrix}
                0 & u\left(1+\frac{1}{z}\right)^2 \\
                v\left(1-\frac{1}{z}\right)^2 & 0
                \end{pmatrix} \dd z
    \]
    with holomorphic non-vanishing $u,v$ such that $4uv=-a^2$, or the transpose of those matrices. All the choices of $\Phi$ lie in the same gauge equivalent class. Indeed, we can show this by applying the conjugate transformation given by the following  
    \[
         G = \begin{pmatrix}
                1 & 0 \\
                0 & f
                \end{pmatrix}, \qquad
        G'=  \begin{pmatrix}
                0 & 1 \\
                1 & 0
                \end{pmatrix}
    \]
    for holomorphic non-vanishing $f$. Note that $E$ is equipped with a natural non-degenerate symmetric pairing $C$ given by
    \[
        C = \begin{pmatrix}
                0 & 1 \\
                1 & 0
                \end{pmatrix}.
    \]
    By \cite[Theorem 1.9]{LiMochizuki2023}, there is a unique harmonic metric $H$ on $E$ that is compatible with $C$. Applying the gauge transformation given by $G$, the symmetric pairing becomes $fC$, and the transported harmonic metric is $\tilde{H}$. Since $\log|f|$ is harmonic, the rescaling metric $|f|^{-1}\tilde{H}$ is still harmonic and is compatible with $C$. Since the trace of $\Phi\Phi^{*H}$ does not change under the rescaling of metric, the pullback of the hyperbolic metric keeps unchanged, and so is $\Psi^*(\sigma+\dd t^2)$. The gauge transformation given by $G'$ does not change $C$, so the conclusion follows directly.
\end{proof}

Combining \cref{prop:intrinsic} and \cref{prop:higgsfield}, we get

\begin{proposition}\label{prop:isometrybyhopf}
    The parameter $a$ determines and is determined by the pullback metric by $\Psi$ on this chosen coordinate.
\end{proposition}

For a horizontal catenoid $\Psi=(F,h)$, the Hopf differential of $F$ is 
\[
    \Hopf(F)=-a^2\left(1-\frac{1}{z^2}\right)^2 \dd z^2.
\]
Suppose it is asymptotic to the vertical plane $\gamma_1 \times \R$ near the end at infinity. We choose $\omega_1$ to be the harmonic map from $\C$ to $\HH^2$ that has image $\gamma_1$ with constant Hopf differential $-a^2\dd z^2$. We require further that $\omega_1$ and $F$ is asymptotic to $\gamma_1$ in the normalized sense by considering $\tilde{\omega}_1(z)= \omega_1(-z)$ instead of $\omega_1$ if necessary. We can assign $\gamma_2$ and $\omega_2$ at $0$ similarly.

By \cref{prop:principaltobounded}, the hyperbolic distance between $F$ and $\omega_1$ is bounded near infinity, and the hyperbolic distance between $F$ and $\omega_2$ is bounded near $0$. Let $X= \C\mathbb{P}^1,\quad D= \{0,\infty\}$. Let $h$ be the harmonic map constructed in \cref{thm:main} with prescribed asymptotic data $\omega_1$ at infinity and $\omega_2$ at $0$. By \cref{rem:absolute-constant}, we have
\[
    d_{\HH^2}(F(z),h(z)) \leq M
\]
for all $z \in \C^*$ and an absolute constant $M>0$. Since the image of $F$ is not contained in a geodesic, \cref{prop:unique} shows that $F=h.$

\appendix
\crefalias{section}{appendix}

\section{A decay estimate for solutions of sinh-Gordon equation}\label{app:decayestimate}

In this section, we give an exponential decay estimate of the solution of sinh-Gordon equation
\[
    \Delta \mu = 2\sinh(2\mu).
\]

\begin{lemma}\label{lem:decay}
    Suppose $\mu$ is smooth on a neighborhood of $\overline{B(0,R)}$ with $R> 2$ and it satisfies
    \[
        \Delta \mu = 2\sinh(2\mu).
    \]
    Then
    \[
        |\mu(0)| + |\nabla\mu(0)| \leq A \e ^{-cR}
    \]
    for some absolute constants $A,c>0$.
\end{lemma}

\begin{proof}
    Consider the following boundary value problem on the unit disk $B(0,1)$:
    \begin{equation*}
        \left\{
        \begin{aligned}
            &\Delta u_n= 2\sinh(2u_n),\\
            &u_n|_{\partial B(0,1)}=n.
        \end{aligned}
        \right.
    \end{equation*}
    The comparison principle gives
    \[
        0<u_n\leq u_{n+1}.
    \]
    Let
    \[
        V(z) = 2-\log(1-|z|^2).
    \]
    A direct calculation gives
    \[
        \Delta V(z) = \frac{4}{(1-|z|^2)^2}.
    \]
    Also,
    \[
        2\sinh(2V) (z)=\frac{1}{(1-|z|^2)^2}(\e^4-\e^{-4}(1-|z|^2)^4) \geq \frac{4}{(1-|z|^2)^2} = \Delta V(z).
    \]
    Furthermore,
    \[
        V(|z|) \to +\infty, \quad |z| \to 1.
    \]
    The comparison principle gives
    \[
        u_n \leq V.
    \]
    Thus the increasing sequence $\{u_n\}$ converges to a smooth solution $u$ which is pointwise positive. 

    Since $u \to +\infty$ as $z \to \partial B(0,1)$, any positive maximum of $|\mu|-u$ must be attained at an interior point $p\in B(0,1)$. We have 
    \[
        |\mu|(p)>0.
    \]
    Observe that $|\mu|$ is smooth near $p$ and 
    \[
        \Delta |\mu| = 2 \sinh(2|\mu|).
    \]
    We obtain
    \[
        0 \geq \Delta(|\mu|-u)(p) = 2(\sinh(|2\mu|)-\sinh(2u)) >0,
    \]
    which is a contradiction. Therefore
    \[
        |\mu|\leq u,
    \]
    and in particular,
    \[
        |\mu(0)|\leq u(0)\leq 2.
    \]
    The same estimate holds for any point $p \in B(0,R-2)$ since that $d(p,\partial B(0,R)) > 1$.

    On $B(0,R-2)$, let
    \[
        g_R(z)= 2\frac{I_0(2|z|)}{I_0(2R-4)}>0,
    \]
    where $I_0$ is the modified Bessel function. $g_R$ solves the equation
    \[
        \Delta g_R = 4g_R.
    \]
    Wherever $|\mu|>0$, we have
    \[
        \Delta |\mu| = 2 \sinh(2|\mu|) \geq 4|\mu|.
    \]
    Applying the comparison principle on $B(0,R-2) \cap \{|\mu|>0\}$, we obtain 
    \[
        |\mu| \leq g_R, \, \mathrm{on}\,B(0,R-2).
    \]
    In particular,
    \[
        |\mu(0)| \leq g_R(0)= \frac{2}{I_0(2R-4)}.
    \]
    Since 
    \[
        I_0(t) \sim \frac{\e^t}{\sqrt{2\pi t}}
    \]
    for large $t$, it follows that
    \[
        |\mu(0)| \leq C\sqrt{R} \e^{-2R} \leq C' \e^{-cR}.
    \]
    The same inequality holds for every $p \in B(0,1)$ since $d(p,\partial B(0,R))>R-2+1=R-1$. 

    It remains to estimate the gradient. We obtain
    \[
        \norm*{\mu}_{L^\infty(B(0,1))} \leq C'\e ^{-cR},
    \]
    and thus
    \[
        \norm*{\Delta\mu}_{L^\infty(B(0,1))} \leq C''\e ^{-cR}.
    \]
    The standard interior gradient estimate yields
    \[
        |\nabla \mu|(0) \leq D(\norm*{\mu}_{L^\infty(B(0,1))} + \norm*{\Delta\mu}_{L^\infty(B(0,1))})  \leq K\e^{-cR}.
    \]
    This completes the proof.
\end{proof}

Suppose $\Omega \subset \C$ is a planar domain that intersects the real axis and
\[
    f:\Omega\to \HH^2
\]
is harmonic with $ \Hopf(f)=\dd z^2.$
Let $\eta=\frac{1}{2}\log H$. Then the pullback metric is 
\[
    f^*\sigma=4\cosh^2{\eta} \dd x^2 + 4 \sinh^2{\eta} \dd y^2.
\]
The Bochner equations for $H$ and $L$ give
\[
    \Delta \eta = \frac{1}{4}(\Delta \log H-\Delta \log L) = H-L = 2\sinh(2\eta).
\]
\cref{lem:decay} shows that, for $p\in \Omega$ and $2<r<d(p,\partial \Omega)$, we have 
    \[
        |\eta(p)|+|\nabla \eta (p)| \leq A \e^{-cr}.
    \]
Now consider the curve $\alpha(x) =f (x,0)$. Let $\mathcal{J}$ denote the complex structure on $\HH^2$, that is, the positively oriented rotation by $\pi/2$ in each tangent space. Define an oriented orthonormal frame $\{e_1,e_2\}$ of the pullback tangent bundle $f^*T\HH^2$ as follows:
\[
    e_1= \frac{f_x}{|f_x|}, \quad e_2=\mathcal{J}e_1.
\]
We have $f_x = 2\cosh(\eta) e_1, f_y=2\sinh(\eta) e_2$. The harmonic map equation is 
\[
    \nabla_{\frac{\partial}{\partial x}}f_x + \nabla_{\frac{\partial}{\partial y}}f_y=0.
\]
We also have
\[
    \nabla_{\frac{\partial}{\partial x}}f_x = 2 \eta_x\sinh(\eta) e_1 + 2\cosh(\eta)\nabla_{\frac{\partial}{\partial x}}e_1
\]
and
\[
    \nabla_{\frac{\partial}{\partial y}}f_y = 2 \eta_y\cosh(\eta) e_2 + 2\sinh(\eta)\nabla_{\frac{\partial}{\partial y}}e_2.
\]
Taking inner product with $e_2$, we obtain:
\[
    \inner{\nabla_{\frac{\partial}{\partial x}}f_x}{e_2} = 2\cosh(\eta) \inner{\nabla_\frac{\partial}{\partial x}e_1}{e_2} = - \inner{\nabla_{\frac{\partial}{\partial y}}f_y}{e_2} = -2 \eta_y \cosh(\eta),
\]
since
\[
    \inner{\nabla_\frac{\partial}{\partial y}e_2}{e_2} = 0.
\]
Thus the geodesic curvature of $\alpha$ is
\[
    k_g(\alpha)= \frac{1}{2\cosh(\eta)} \inner{\nabla_\frac{\partial}{\partial x}e_1}{e_2} = \frac{-\eta_y}{2\cosh(\eta)}.
\]
By \cref{lem:decay}, we obtain:
\begin{equation}\label{eq:kg-bound}
    |k_g(\alpha)|(p) \leq \frac{1}{2}A\e^{-cr}
\end{equation}
if $2<r<d(p,\partial\Omega)$. We emphasize that the constant $A,c>0$ are independent of $f$, by \cref{lem:decay}. 

\begin{remark}\label{rem:dz2}
    For a harmonic map $f:\C\to \HH^2$ with constant Hopf differential $\dd z^2$, $\eta$ is an entire smooth solution of sinh-Gordon equation on $\C$. Applying \cref{lem:decay} and choosing $R$ arbitrarily large, we show that $\eta$ is constantly $0$, i.e., $H=L=1.$
    Thus
    \[
        f^*\sigma=4\dd x^2.
    \]
    The map $f$ is independent of $y$, and the harmonic map equation becomes the geodesic equation. Thus $f|_{\R}$ is a geodesic of constant speed $2$.
\end{remark}

\section{A Hessian inequality}\label{app:hess}
In this section, we introduce a Hessian inequality and use it to prove that $\Hopf(h)$ is meromorphic at each of $p_1,...,p_k$. The following theorem is the special case $\kappa=0$ of \cite[Lemma 2.2]{ChenJostWang2013}:
\begin{theorem}
    Denote $d$ the distance function on $\HH^2$. For $x,y \in \HH^2$ and $V \in T_x \HH^2, W \in T_y\HH^2$, we have
    \[
        \nabla^2\left(\frac{1}{2}d^2\right)_{(x,y)}\big((V,W),(V,W)\big) \geq d(x,y)\int_0^{d(x,y)} \norm*{\dot{J}(t)}^2\dd t.
    \]
    Here $J$ is the Jacobi field along the unit speed geodesic $l$ connecting $x$ and $y$ with $J|_x = V, J|_y = W$, and $\dot{J}$ is the covariant derivative of $J$ along $l$.
\end{theorem}

Suppose $l(0)=x, l(d(x,y))=y$. Let $T_t$ denote the parallel transport along $l$ from $T_{l(t)}\HH^2$ to $T_{l(0)}\HH^2$. We have
\[
    T_{d(x,y)}W-V = \int_0^{d(x,y)} T_t\dot{J}(t) \dd t.
\]
Cauchy-Schwarz inequality gives:
\begin{equation}\label{eq:hess}
    \norm*{T_{d(x,y)}W-V}^2 \leq d(x,y) \int _0 ^{d(x,y)} \norm*{\dot{J}(t)}^2\dd t \leq \nabla^2\left(\frac{1}{2}d^2\right)_{(x,y)}\big((V,W),(V,W)\big).
\end{equation}

\begin{proposition}\label{prop:pole}
    Suppose $\Phi=\varphi(z)\dd z^2$ is a degree $d$ polynomial quadratic differential with associated Scherk map $\omega$. Let $h: \Omega:=\{|z|>1\} \to \HH^2$ be harmonic and satisfies $d_{\HH^2}(h(z),\omega(z)) \leq M$ for a constant $M>0$ and all $z \in \Omega$, then the order of $\Hopf(h)$ at $\infty$ is $-d-4$.
\end{proposition}

\begin{proof}
    Let $u(z) = d^2_{\HH^2}(h(z),\omega(z))/2$. For each $z \in \Omega$, choose an orthonormal basis $\{e_1,e_2\}$ of $T_z\Omega$. In \eqref{eq:hess}, replacing $V$ by $\dd h|_{z}(e_i)$ and replacing $W$ by $\dd \omega|_z(e_i)$, we obtain:
    \[
        \nabla^2 \left(\frac{1}{2}d^2\right)\Big(\big(\dd h(e_i),\dd\omega(e_i)\big),\big(\dd h(e_i),\dd\omega(e_i)\big)\Big) \geq \norm*{\dd h(e_i) - T \dd \omega(e_i)}^2.
    \]
    Here $T$ is a parallel transport.
    Summing over $i=1,2$, we obtain pointwisely:
    \[
        \Delta u \geq \norm*{\dd h - T \dd \omega}^2
    \]
    Choose $\eta \in C_c^{\infty}(B(0,2))$ such that
    \[
        0\leq \eta \leq 1, \quad \eta \equiv 1\,\mathrm{on}\,B(0,1)
    \]
    and
    \[
        \int_{B(0,2)} |\Delta \eta| \leq C.
    \]
    For $z_0 \in \Omega$ with $|z_0|>3$, let $\eta_{z_0}(z)= \eta\big(4(z-z_0)/|z_0|\big)$ be in $C^\infty_c(B(z_0,|z_0|/2))$. Note that
    \[
        \int_{B(z_0,|z_0|/2)} |\Delta \eta_{z_0}| \leq C
    \]
    still holds. We have
    \begin{align*}
        \int_{B(z_0,|z_0|/4)} \norm*{\dd h - T \dd \omega}^2 &\leq \int_{B(z_0,|z_0|/2)} \eta_{z_0} \norm*{\dd h - T \dd \omega}^2\\
        & \leq \int_{B(z_0,|z_0|/2)} \eta_{z_0} \Delta u =\int_{B(z_0,|z_0|/2)} u \Delta\eta_{z_0} \\
        & \leq \int_{B(z_0,|z_0|/2)} u |\Delta\eta_{z_0}| \\
        & \leq \frac{1}{2}CM^2.
    \end{align*}
    Thus    
    \begin{align*}
        \int_{B(z_0,|z_0|/4)} \norm*{\dd h}^2 &\leq CM^2 +2 \int_{B(z_0,|z_0|/4)} \norm *{\dd \omega} ^2 \\
        & \leq D(1+|z_0|^{d+2})
    \end{align*}
    for some constant $D$ independent of $z_0$, since $e(\omega) \leq J(\omega) + 2 \norm*{\Phi}$, the integral of $J$ over $\C$ is a constant, and $\Phi$ is a polynomial. We then have
    \begin{align*}
        \abs*{\Hopf(h)-\Phi}^2(z_0) &\leq D' |z_0|^{-4} \left(\int_{B(z_0,|z_0|/4)} \abs*{\Hopf(h)-\Phi} \right)^2\\
        & \leq D''|z_0|^{-4} \left(\int_{B(z_0,|z_0|/4)} \norm*{\dd h -T \dd \omega} (\norm*{\dd h}+\norm*{\dd \omega})\right)^2 \\
        & \leq D'' |z_0|^{-4} \int_{B(z_0,|z_0|/4)} \norm*{\dd h -T \dd \omega}^2 \int_{B(z_0,|z_0|/4)}  (\norm*{\dd h}+\norm*{\dd \omega})^2 \\
        & \leq K|z_0|^{-4}(1+|z_0|^{d+2})
    \end{align*}
    for absolute constant $D',D'',K>0$. Thus 
    \[
        \abs*{\Hopf(h)-\Phi}(z_0) \leq K'(|z_0|^{-2}+|z_0|^{d/2-1}).
    \]
    All the constants above are independent with $z_0$, therefore $\Hopf(h)$ and $\Phi$ have the same pole order at infinity, as desired. 
\end{proof}

\section{Polygonal exhaustion}\label{app:polygonexhaustion}
In this section, we give the construction of a polygonal exhaustion used in the proof of \cref{prop:principaltobounded}. This construction is similar to \cite[Definition 2.27 and Proposition 2.29]{Gupta2021Harmonic}.

Let $q$ be a holomorphic quadratic differential on $\{|z|>1\}$ which has a pole of order $n\geq 4$ at infinity. We assume $n$ is even, otherwise we do the construction on a double cover. This makes sense since in \cref{sec:principal-bound},the principal part of a quadratic differential of odd order is defined on a double cover. We also assume $q$ extends holomorphically to a neighborhood of $\{|z|\geq 1\}$.

We choose $R>0$ sufficiently large so that $q$ never vanishes on $\{|z|>R\}$. In $\{|z|>R\}$, we can choose $m=n-2$ horizontal half planes $H_1,\dots , H_m$ and $m$ vertical half planes $V_1,\dots ,V_m$ such that, their union is a neighborhood of infinity, each $H_i$ intersects exactly $V_i$ and $V_{i-1}$ and each $V_i$ intersects exactly $H_i$ and $H_{i+1}$. We arrange these half planes in cyclic order
\[
    H_1,V_1,H_2,V_2,\dots,H_m,V_m.
\]

We first choose a natural coordinate $z_1=x_1+\ii y_1$ on $H_1$ such that $z_1(H_1)$ is the upper half plane. Let $\varepsilon_1=1$ and $b_1=0$. Suppose $\partial V_1 \cap H_1$ is the line $\{x_1=a_1\} \cap \{y_1>0\}$. If $V_1 \cap H_1$ is on the left side of $\partial V_1$, set $\tau_1=-1$. If $V_1 \cap H_1$ is on the right side of$\partial V_1$, set $\tau_1=1$. Choose the natural coordinate $w_1=u_1+\ii v_1$ on $V_1$ such that $w_1=z_1$ on $V_1 \cap H_1$. Thus $V_1$ is parameterized by $\{\tau_1 u_1>\tau_1a_1\}$.

Choose the natural coordinate $z_2=x_2+\ii y_2$ on $H_2$ such that it agrees with $w_1$ on $H_2 \cap V_1$. Similarly, suppose $\partial H_2 \cap V_1$ is the line $\{v_1=b_2\}\cap \{\tau_1u_1>\tau_1a_1\}$. If $H_2 \cap V_1$ lies above $\partial H_2$, set $\varepsilon_2=1$. If $H_2 \cap V_1$ lies below $\partial H_2$, set $\varepsilon_2=-1$. Thus $H_2$ is parameterized by $\{\varepsilon_2y_2>\varepsilon_2b_2\}$. Since $H_1$ does not intersect $H_2$, we must have $\varepsilon_2=-\varepsilon_1$.

We continue successively, and end this process when all the horizontal and vertical half planes are parameterized by $z_k(H_k)=\{\varepsilon_k y_k>\varepsilon_kb_k\}$ and $w_k(V_k)=\{\tau_ku_k>\tau_ka_k\}$. The signs also alternate, i.e. $\varepsilon_{k+1}=-\varepsilon_k$ and $\tau_{k+1}=-\tau_k$. Note that
\begin{align*}
    &\im(w_k(\partial H_k\cap V_k))=b_k,\\
    &\im(w_k(\partial H_{k+1}\cap V_k))=b_{k+1},\\
    &\re(z_k(\partial V_k\cap H_k))= a_k,\\
    &\re(z_k(\partial V_{k-1}\cap H_k))= a_{k-1}.
\end{align*}
Here we need to deal with the overlap $V_m \cap H_1$ separately. Suppose 
\[
    \im(w_m(\partial H_1 \cap V_m)) = b_{m+1},
\]
and 
\[
    \re(z_1(\partial V_m \cap H_1)) = a_0.
\]
We make no further adjustment to the parameter $z_1$ by $b_{m+1}$ or $a_0$.

We define $H_k(T)= z_k^{-1}(\{\varepsilon_k y_k>\varepsilon_kb_k+T\})$ and $V_k(T)=w_k^{-1}(\{\tau_ku_k>\tau_ka_k+T\})$. The union $U(T)=\cup_{i=1}^m \left(H_i(T)\cup V_i(T)\right)$ is a neighborhood of infinity and its boundary is
\[
    \Gamma(T)=\partial U(T)= \alpha_1(T) \cup \beta_1(T) \cup \dots \cup \alpha_m(T) \cup \beta_m(T),  
\]
where $\alpha_i(T)=\partial H_i(T) \cap \Gamma(T)$ is a horizontal segment, and $\beta_i(T) = \partial V_i(T) \cap \Gamma(T)$ is a vertical segment. Thus $\Gamma(T)$ can be viewed as a "polygon" whose edges are alternating horizontal and vertical lines.

By the choices of the coordinates, we obtain
\begin{align*}
    &\im[w_k(\partial H_k(T)\cap V_k(T))]=b_k+\varepsilon_k T,\\
    &\im[w_k(\partial H_{k+1}(T)\cap V_k(T))]=b_{k+1}+\varepsilon_{k+1} T,\\
    &\re[z_k(\partial V_k(T)\cap H_k(T))]= a_k+\tau_k T,\\
    &\re[z_k\left(\partial V_{k-1}(T)\cap H_k(T)\right)]= a_{k-1}+ \tau_{k-1}T.
\end{align*}
Let 
\[
    M= \max\{|a_1|,|b_1|,\dots |a_m|, |b_m|, |a_0|, |b_{m+1}| \}.
\]
Then the $q$-length of $\alpha_i(T)$ and $\beta_i(T)$ satisfies
\begin{equation}\label{eq:horizontalO1}
    2T-2M \leq l_q(\alpha_i(T)) \leq 2T +2M
\end{equation}
and
\begin{equation}\label{eq:verticalO1}
    2T-2M \leq l_q(\beta_i(T)) \leq 2T +2M.
\end{equation}
Also observe that the $q$-distance between $\Gamma(T)$ any fixed point $z_0 \in \{1<|z|\leq R\}$  satisfies:
\[
    T \leq d_q(\Gamma(T),z_0) \leq D_1T + D_2.
\]
Here $D_1,D_2>0$ are constants independent of $T$ and $z_0$. Take a increasing sequence $\{T_j\} \to + \infty$, such that $\Gamma(T_{j+1})$ lies outside $\Gamma(T_j)$. Then the sequence of region bounded by $\Gamma(T_j)$ and the unit circle is the desired exhaustion.  

\bibliographystyle{amsalpha}
\bibliography{ref}

\end{document}